\documentclass[11pt,reqno]{amsart}

\usepackage[shortlabels]{enumitem}

\usepackage{color}
\usepackage{amsmath, amsthm, amssymb}
\usepackage{amsfonts}

\usepackage[ansinew]{inputenc}
\usepackage[dvips]{epsfig}
\usepackage{graphicx}
\usepackage[english]{babel}
\usepackage{hyperref}
\theoremstyle{plain}
\newtheorem{theorem}{Theorem}[section]
\newtheorem{corollary}[theorem]{Corollary}
\newtheorem{lemma}[theorem]{Lemma}
\newtheorem{proposition}[theorem]{Proposition}

\theoremstyle{definition}
\newtheorem{definition}[theorem]{Definition}

\theoremstyle{remark}
\newtheorem{remark}[theorem]{Remark}

\numberwithin{equation}{section}

\newcommand{\average}{{\mathchoice {\kern1ex\vcenter{\hrule height.4pt
				width 6pt depth0pt} \kern-9.7pt} {\kern1ex\vcenter{\hrule
				height.4pt width 4.3pt depth0pt} \kern-7pt} {} {} }}

\def\R{\mathbb{R}}
\def\S{\mathbb{S}}

\def\N{\mathbb{N}}

\begin{document}
	
	\title[$C^{2,\alpha}$ regularity of the free boundary]{$C^{2,\alpha}$ regularity of   free boundaries in\\ parabolic non-local obstacle problems}
	
	\author{Teo Kukuljan}
	\address{Universitat de Barcelona,
		Departament de Matematiques i Inform\`atica, Gran Via de les Corts Catalanes 585,
		08007 Barcelona, Spain.}
	\email{tkukuljan@ub.edu}
	
	\thanks{The author has received funding from the European Research Council (ERC) under the Grant Agreement No 801867.
	The paper was written under supervision of Xavier Ros-Oton, we thank him for proposing many ideas and approaches.}
	\thanks{Data sharing not applicable to this article as no datasets were generated or analysed during the current study.}
	
	\begin{abstract}
		We study the regularity of the free boundary in the parabolic obstacle problem for the fractional Laplacian $(-\Delta)^s$ (and more general integro-differential operators) in the   regime $s>\frac{1}{2}$. We prove that once the free boundary is $C^1$ it is actually $C^{2,\alpha}$. 
		
		To do so, we establish a boundary Harnack inequality in $C^1$ and $C^{1,\alpha}$ (moving) domains, providing that the quotient of two solutions of the linear equation, that vanish on the boundary, is  as smooth as the boundary.
		
		As a consequence of our results we also establish for the first time optimal regularity of such solutions to nonlocal parabolic equations in moving domains. 
	\end{abstract}

	\maketitle
	
\section{Introduction}	

	The fractional obstacle problem is the following
	\begin{align*}
		\min\{ (-\Delta)^s u , u-\varphi \} &= 0 \quad\text{ in } \R^n,\\
		\lim _{|x|\to\infty}u(x) &= 0,
	\end{align*}
	where $\varphi\colon\R^n\to\R$ is a given function called the obstacle and 
	$$(-\Delta)^s u (x) = c_{n,s}p.v.\int_{\R^n}(u(x)-u(x+y))|y|^{-n-2s}dy,$$
	for some $s\in(0,1)$. The operator $(-\Delta)^s$ is called the fractional Laplacian.
	This problem (and its parabolic version) arises in the study of the optimal stopping problems for stochastic processes for example when modelling the prices of American options. For more information see \cite{CT04}. 
	
	The study of the fractional obstacle problem was initiated by Silvestre   in \cite{S07} and by Caffarelli, Salsa and Silvestre in \cite{CSS08}. Since then a lot of effort was put in studying this problem and is nowadays quite well understood. 
	The interest of study in the obstacle problems is twofold. On one hand we are interested in regularity of the solutions, and on the other one we want to understand the set $\partial\{u>\varphi\}$, called the free boundary. 
	In \cite{CSS08} they showed the optimal $C^{1+s}$ regularity of solutions, moreover they proved that at any free boundary point $x_0\in\partial\{u>\varphi \}$ exactly one of the following two statements holds
	$$\begin{array}{lcccccl}
	\text{(i) }\quad 0&<&cr^{1+s}&\leq& \sup_{B_r(x_0)} (u-\varphi)&\leq& Cr^{1+s}\\
	\text{(ii) } &&0&\leq&\sup_{B_r(x_0)} (u-\varphi)&\leq &Cr^{2}
	\end{array}\quad\quad\forall r\in(0,r_0).$$
	The points satisfying $\text{(i)}$ are called regular points, they form an open subset of the free boundary, and the free boundary is $C^{1,\alpha}$  there for some $\alpha>0$. 
	Later on it was established  that near regular points the free boundary is in fact $C^\infty$ if the obstacle is $C^\infty$ (see \cite{JN17,KRS19}). For results about the degenerate/singular points - the ones satisfying $\text{(ii)}$ - we refer to \cite{FR21,F20,GR19} and references therein. 
	
	Many methods used for establishing these results strongly depend on the tools only available in the case of the fractional Laplacian. Therefore it was also challenging to extend the above mentioned results to a more general class of integro-differential operators. The dichotomy of the regular and degenerate points, as well as the $C^{1,\alpha}$ regularity of the free boundary near regular points was proved in \cite{CRS17}, and the higher order regularity of the free boundary in \cite{AR20}.
	
	Much less is known in the parabolic version of the problem:
	\begin{align*}
	\min\{ \partial_t + (-\Delta)^s u , u-\varphi \} &= 0 \quad\text{ in } \R^n\times(0,T),\\
	u(\cdot,0) &= \varphi.
	\end{align*}
	Notice that the nature of the problem strongly depends on the value of the parameter $s$. For $s>\frac{1}{2}$ the fractional Laplacian is the leading term (the subcritical regime), while for $s<\frac{1}{2}$ the time derivative is (the supercritical regime). 
	The regularity of solutions was first addressed by Caffarelli and Figalli in \cite{CF13}, where they proved that the solutions are $C^{1+s}_x$ in space and  $C^{\min(\frac{1+s}{2s},2)-\varepsilon}_t$ in time. 
	Later on in \cite{BFR18} Barrios, Figalli and Ros-Oton showed that when $s>\frac{1}{2}$ we have the analogous dichotomy as in the elliptic case, and that the free boundary is $C^{1,\alpha}$ in space and time near regular free boundary points. 
	Furthermore, very recently in \cite{RT21} Ros-Oton and Torres-Latorre improved the regularity of solutions in the supercritical regime $s<\frac{1}{2}$ to $C^{1,1}$ in space and time, which is optimal. Finally in  a forthcoming paper \cite{FRS22} Figalli, Ros-Oton and Serra  extended the results from \cite{BFR18} to a more general class of operators as well as to the case $s=\frac{1}{2}$. 	
	
	Despite all these results nothing is known though about the higher regularity of the free boundary. An open question that seems quite challenging is the following: 
	$$\textit{Is it true that the free boundary is $C^\infty$ (at least in space) near regular points? }$$
	
	The goal of this paper is to study this question for $s>\frac{1}{2}$ and establish that near regular points the free boundary is $C^{2,\alpha}$. 
	We consider the class of non-local operators of the form
	\begin{align}\label{eq:operatorForm}
	\begin{split}
	Lu(x,t) = p.v.\int_{\R^n}(u(x,t)-u(y,t)) K(x-y)dy\\
	\lambda |y|^{-n-2s} \leq K(y)\leq \Lambda |y|^{-n-2s},\text{ }y\in\R^n,\quad\text{ }K\text{ is even,  homogeneous},
	\end{split}
	\end{align}
	for $0<\lambda\leq \Lambda$, called the ellipticity constants.	
	Our main result reads as follows. (We refer to Section \ref{sec:notation} for definition of the parabolic H\"older spaces $C^\beta_p$.)
	
	\begin{theorem}\label{thm:1.1}
		Let $s\in(\frac{1}{2},1)$. Let $L$ be as in \eqref{eq:operatorForm}, with its kernel $K\in C^5(\S^{n-1})$. Let $u$ be  a solution of 
		$$\begin{array}{rcll}
		\min\{(\partial_t+L)u,u-\varphi \}&=&0\quad &\text{ in }\R^n\times(0,T)\\
		u(\cdot,0)&=&\varphi\quad &\text{ in }\R^n,
		\end{array}$$
		with $\varphi\in C^{4}(\R^n)$. 
		
		Then the free boundary $\partial  \{u>0\}$ is $C^{2+\alpha}_p$ near regular free boundary points for some $\alpha>0$.
	\end{theorem}
	
	It remains an open problem to improve the regularity from $C^{2,\alpha}$ to $C^\infty$; see Remark~\ref{rem:1} below.
	
	To prove this result, we exploit the fact that the normal to the free boundary can be expressed with the quotients of partial derivatives of $u-\varphi$, see \cite{AR20,DS15,DS16}. Hence we closely study the boundary behaviour of solutions to 
	\begin{equation}\label{eq:equation}
		\left\lbrace\begin{array}{rcll}
		(\partial_t  + L)w &=& f& \text{in }\Omega\cap Q_1\\
		w&=&0&\text{in }\Omega^c\cap Q_1,
		\end{array}\right. 
	\end{equation}
	for some open set $\Omega\subset\R^{n+1}$ and $Q_1 = \{(x,t)\in\R^{n+1}:\text{ }|x|<1,|t|<1\}$, as the mentioned derivatives solve such equation. There was not much attention put into the studying of the boundary regularity of solutions in \textit{moving} domains, the only known results to the bests of our knowledge (\cite{BWZ17,FR17,G19,RV18}) consider cylindric types of domains. We provide that the solutions grow like distance to the boundary to the power $s$, which in combination with interior regularity estimates gives $C^s_p$ regularity up to the boundary. 	
	Refining the argument, and comparing solutions directly amongst each other, we are able to establish the boundary Harnack inequality of the following type.
	
	\begin{theorem}\label{thm:bdryHarnack}
		Let $s\in(\frac{1}{2},1)$. Let $\Omega\subset \R^{n+1}$ be $C^{1}_p$ in $Q_1$. Let $\gamma,\varepsilon>0$. Let $L$ be an operator of the form \eqref{eq:operatorForm}, with kernel $K\in C^{1-s}(\S^{n-1})$. Assume $u_i\in C_p^{\gamma}(Q_1)\cap L^\infty(\R^n\times(-1,1))$, $i\in\{1,2\}$, solve
		$$\left\lbrace\begin{array}{rcll}
		(\partial_t + L)u_i &=& f_i& \text{in }\Omega\cap Q_1\\
		u_i&=&0&\text{in }\Omega^c\cap Q_1,
		\end{array}\right. $$
		with $\left[ f_i\right]_{C^{1-s}_p(\Omega\cap Q_1)}\leq 1$, and $\left[ u_i\right] _{C^{\frac{1-s}{2s}}_t(\R^n\times(-1,1))}\leq 1$. Assume also that $u_2\geq c_0 d^s$, for some $c_0>0$.
		Then  
		$$\left|\left| \frac{u_1}{u_2}\right|\right| _{C_p^{1-\varepsilon}(\overline{\Omega}\cap Q_{1/2})}\leq C,$$
		where $C>0$ depends only on $n,s,\varepsilon,c_0,G_0,\|K\|_{C^{1-s}(\S^{n-1})}$ and ellipticity constants.
		
		If additionally $\Omega$ is $C^\beta_p$ in $Q_1$ for some $\beta>1+s$, $K\in C^{2\beta+1}(\S^{n-1})$,
		$\left[ f_i\right]_{C^{s}_p(\Omega\cap Q_1)}\leq 1$, and $\left[ u_i\right]_{C^{\frac{1}{2 }}_t(\R^{n}\times(-1,1))}\leq 1$, then 
		$$ \left|\left| \frac{u_1}{u_2} \right|\right|_{C^{2s-\varepsilon}_p(\overline{\Omega}\cap Q_{1/2})}  \leq C,$$
		with $C$ additionally depending on $\|K\|_{C^{2\beta+1}(\S^{n-1})}$.
	\end{theorem}

	To obtain the boundary Harnack inequality, we develop expansions of the form $u_1-Qu_2$ at boundary points of orders up to $3s-\varepsilon$, where $Q$ is a polynomial in space variables of degree $1$. In combination with interior regularity results this yields estimates for the quotient as well.  
	\begin{remark}\label{rem:1}
%		This is also one of the reasons why we only get $C^{2,\alpha}_p$ regularity of the free boundary in Theorem \ref{thm:1.1}. Namely, the derivatives of the solution to the obstacle problem are globally roughly only $C^{\min(s,s^{-1}-1)}_t$ (see \cite[Corollary 1.6]{FRS22}), which allows interior estimates of order $\min(2s^2,2-2s) + 2s$. Therefore with the method we use, we can get the boundary Harnack estimate of order at most  $\min(2s^2,2-2s) + s = \min(s(2s+1),2-s)$, which is smaller than $2$, and hence we can not deduce that the free boundary is better than $C^{2,\alpha}_p$. We expect this to be an unbreachable obstacle when dealing with the regularity of the free boundary in time, while we expect the regularity in space to be infinite. We do not have any examples or counterexamples to support this claim. 
	Note that in the parabolic non-local setting the interior regularity estimates require global time regularity of solutions (see \cite[Lemma 7.1]{FR17}).
	This is also one of the reasons why we only get $C^{2,\alpha}_p$ regularity of the free boundary in Theorem \ref{thm:1.1}, for some $\alpha>0$. Namely, the time derivative of the solution to the obstacle problem is globally $C^{\frac{\gamma}{2s}}_t$, for some $1-s<\gamma < 2-s$ (see \cite{BFR18}, \cite{CF13} and \cite[Corollary 1.6]{FRS22}). Hence we can only use the interior estimates for orders up to $\gamma + 2s$. Therefore with the method we use, we can get the boundary Harnack estimate of order at most  $\gamma + s $, which is smaller than $2$, and hence we can not deduce that the free boundary is better than $C^{2, \alpha}$. %We expect this to be an unbreachable obstacle when dealing with the regularity of the free boundary in time.
	\end{remark}
%TODO: correct this!!
	
	It remains an open problem to decide whether the free boundary is $C^\infty$ near regular points or not.

	As a consequence of our result we also obtain the optimal H\"older regularity of $u$ and of $u/d^s$, thus extending the results of \cite{RV18} to the case of moving domains. More precisely, we prove that any solution of \eqref{eq:equation} satisfies
	$$u\in C^s_{x,t}\quad\quad \text{and}\quad\quad \frac{u}{d^s}\in C^{2s-1}_p,$$
	see Corollary \ref{cor:CsRegularity}. Here $d$ is a regularized space-time distance to the boundary. 

	\subsection{Organisation of the paper}
	We start the body with presenting the notation in Section \ref{sec:notation}. 
	In Section \ref{sec:bdryReg} we provide the results regarding the boundary regularity of solutions to \eqref{eq:equation}, and we prove Theorem \ref{thm:bdryHarnack}. 
	In Section \ref{sec:obstProb} we prove Theorem \ref{thm:1.1}. In the last section of the body, Section \ref{sec:Lds}, we establish results about operator evaluation of the distance function to the power $s$, which is required in some proofs from Section \ref{sec:bdryReg}. 
	At the end there is an appendix where we prove technical auxiliary results, to lighten the body of the paper.
	
\section{Notation and preliminary definitions}\label{sec:notation}
The ambient space is $\R^{n+1} = \R^n\times\R,$ where the first $n$ coordinates we denote by  $x = (x_1,\ldots,x_n)$ and the last one with $t$. Sometimes we furthermore split $x=(x',x_n)$, for $x'=(x_1,\ldots,x_{n-1})$. 
Accordingly we use the multi-index notation $\alpha\in\N_0^{n+1}$, $\alpha = (\alpha_1,\ldots,\alpha_n,\alpha_t)$, $|\alpha|_p=\alpha_1+\ldots+\alpha_n+2s\alpha_t$, and furthermore
$$\partial^\alpha = \left(\frac{\partial}{\partial x_1}\right)^{\alpha_1}\circ\ldots\circ \left(\frac{\partial}{\partial x_n}\right)^{\alpha_n}\circ \left(\frac{\partial}{\partial t}\right)^{\alpha_t}.$$
The gradient operator $\nabla$, differential operator $D$ and Laplace operator $\Delta$  are taken only in $x$ variables.

For $\Omega\subset\R^{n+1}$ we denote the time slits with $\Omega_t = \{x\in\R^n;\text{ }(x,t)\in\Omega \}$ and the distance function to the boundary in space directions only with $d_t(x) = d_x(x,t) = \inf_{z\in\partial\Omega_t}|z-x|,$ while with $d$ we denote  the (regularized) distance to the boundary of the domain in space-time. We also write $|(x,t)-(x',t')|_p = |x-x'| + |t-t'|^{\frac{1}{2s}}$.

We denote by  $Q_r(x_0,t_0)$ the following cylinder of radius $r$ centred at $(x_0,t_0)$, 
$$Q_r(x_0,t_0)= B_r(x_0)\times (t_0-r^{2s},t_0+r^{2s}).$$ 
When $(x_0,t_0) = (0,0)$, we denote it simply $Q_r$. 
We denote $Q_r' = \{|x'|<r\}\times (t_0-r^{2s},t_0+r^{2s})$.

Finally, $C$ indicates an unspecified constant not depending on any of the relevant quantities, and whose value is allowed to change from line to line. We make use of sub-indices whenever we will want to underline the  dependencies of the	constant.

We define the parabolic H\"older seminorms of order $\alpha>0$ in the following way
\begin{definition}
	Let $\Omega$ be an open subset of $\R^{n+1}$ and let $s\in(\frac{1}{2},1)$. For $\alpha\in(0,1]$ we define the parabolic H\"older seminorm of order $\alpha$ as follows
	$$\left[ u\right]_{C^\alpha_p(\Omega)}= \sup_{(x,t),(x',t')\in\Omega}\frac{|u(x,t)-u(x',t')|}{|x-x'|^\alpha+|t-t'|^{\frac{\alpha}{2s}}},$$
	and
	$$\left[ u\right]_{C^{\alpha}_t(\Omega)} = \sup_{(x,t),(x,t')\in\Omega}\frac{|u(x,t)-u(x,t')|}{|t-t'|^{\alpha}}.$$
	If $\alpha\in(1,2s]$, we set
	\begin{align*}
	\left[ u\right]_{C^\alpha_p(\Omega)}&=  \left[ \nabla u\right]_{C^{\alpha-1}_p(\Omega)}+\left[ u\right]_{C^{\frac{\alpha}{2s}}_t(\Omega)}.
	\end{align*}
	For bigger exponents $\alpha>2s$, we set
	$$\left[ u\right]_{C^\alpha_p(\Omega)}  = \left[ \nabla u\right]_{C^{\alpha-1}_p(\Omega)}+\left[\partial_t u\right]_{C^{\alpha-2s}_p(\Omega)}.$$
	We furthermore define the parabolic H\"older space of order $\alpha$
	$$C^\alpha_p(\Omega) = \{f;\text{ }\left[ f\right]_{C^\alpha_p(\Omega)}<\infty,\text{ }D^k\partial_t^l f \in C^0(\Omega) \text{ whenever }k+2sl\leq \alpha \}.$$
\end{definition}

We are now able to define what means that a set is $C^\alpha_p$.

\begin{definition}\label{definition1.1}
	Let $\Omega\subset \R^{n+1}$ with $0\in\partial\Omega$, $r>0$, and $\alpha>0$. We say that $\Omega$ is $C^\alpha_p$ in $Q_r$, if $\partial\Omega\cap Q_r$ is a graph over the $n$-th coordinate of some function $g\in C^\alpha_p(Q_r')$, with $\left[ g\right]_{C^\alpha_p(Q_r')}\leq G_0$, for some $G_0>0$. 
\end{definition}

In the paper we work with a distance function to the boundary. Since we need it to have more regularity than just the euclidean distance, we give the precise statement in the following definition. 

\begin{definition}\label{def:generalisedDistance}
	Let $\Omega$ be $C^{\beta}_p$ in $Q_1$, for some $\beta>1.$ 
	We denote by  $d$ a function satisfying
	$$d\in C^{\beta}_p(\R^{n+1})\cap C^\infty_x(\{d>0\}),\quad C^{-1}\operatorname{dist}(\cdot,\partial\Omega\cap Q_1)\leq d\leq C \operatorname{dist}(\cdot,\partial\Omega\cap Q_1) \text{ in }\Omega\cap Q_1,$$
	$$|D^kd|\leq Cd^{\beta-k},\quad \text{in }\{d>0\},\text{ for all }k> \beta .$$
	Any such function is called a \textit{regularized distance}.	 
	
	Furthermore we denote by  $\psi$ a diffeomorphism $\psi\colon \R^{n+1}\to\R^{n+1}$, such that it holds $\psi(\Omega\cap Q_1) = Q_1\cap\{x_n>0\}$, %$\psi|_{Q_2^c} = \operatorname{id}$ 
	and that  $\psi\in C^\beta_p(\R^{n+1})\cap C^2_x(\{d>0\})$, with
	$$|D^k \psi|\leq C d^{\beta-k},$$
	for all $k>\beta$, in $\{d>0\}. $ 
	
	The construction of $d$ and $\psi$ is provided in Lemma \ref{lem:generalisedDistanceExsistence}.
\end{definition}

%TODO: what is with the being identity outside Q_2 (I think one needs it in Ld^s, but am not sure - maybe just being close to identity, which I think I get in lemma A1...)

\section{Boundary regularity in moving domains}\label{sec:bdryReg}

	In this section we prove the results regarding boundary regularity of solutions to \eqref{eq:equation}. We follow the ideas from \cite{FR17,RV18}, but in our setting the domain does not need to be cylindric. The main tool is using contradiction arguments in combination with blow-up techniques. This is why we work with domains that are at least $C^1_p$ near the origin, since they give a half-space after blowing up.
	Let us stress that for all the estimates we need to assume that the solutions are already H\"older continuous with some (small) positive exponent. As we use these estimates on the derivatives of the solution to the obstacle problem, which is $C^{1,\alpha}$, this does not cause issues.
	
	\subsection{H\"older estimates}

	We begin with establishing a-priori boundary estimates for orders smaller than $s$. 

	\begin{proposition}\label{prop:holderBoundaryEstimate}
		Let $s\in(\frac{1}{2},1)$. Let  $\Omega\subset \R^{n+1}$ be $C^{1}_p$ in $Q_1$. Let $\gamma\in(0,s)$ and $\varepsilon>0$. Let $L$ be an operator of the form \eqref{eq:operatorForm}. Assume $u\in C_p^{\gamma}(Q_1)$ is a solution of
		$$\left\lbrace\begin{array}{rcll}
		(\partial_t  + L)u &=& f& \text{in }\Omega\cap Q_1\\
		u&=&0&\text{in }\Omega^c\cap Q_1,
		\end{array}\right. $$
		with $fd^{s}\in L^\infty(\Omega\cap Q_1)$.
		Then 
		$$\left[ u\right]_{C_p^{\gamma}( Q_{1/2})}\leq  C\left(||fd^{s}||_{L^\infty(\Omega\cap Q_1)} +\sup_{R>1}R^{\varepsilon-2s} ||u||_{L^\infty(B_R\times (-1,1))}\right).$$
		The constant $C$ depends only on $n,s,\gamma,\varepsilon,G_0$ and ellipticity constants.
	\end{proposition}

	We prove it in two steps. First we establish a weaker estimate of similar type in the following lemma and then prove that it in fact implies the wanted inequality.

\begin{lemma}\label{lem:fakeBoundaryReg}
	Let $s\in(\frac{1}{2},1)$. Let $\Omega\subset \R^{n+1}$ be $C^1_p$ in $Q_1$. Let $\varepsilon>0$. Let $L$ be an operator of the form \eqref{eq:operatorForm} with the kernel $K$. Assume $u\in C_p^{s-\varepsilon}(\R^n\times (-1,1))$ is a solution of
	$$\left\lbrace\begin{array}{rcll}
	(\partial_t  + L)u &=& f& \text{in }\Omega\cap Q_1\\
	u&=&0&\text{in }\Omega^c\cap Q_1,
	\end{array}\right. $$
	with $fd^{s}\in L^\infty (\Omega\cap Q_1)$.
	Then for every $\delta>0$ there exists $C>0$ so that
	$$\left[ u\right]_{C_p^{s-\varepsilon}(\Omega\cap Q_{1/2})}\leq \delta\left[ u\right]_{C_p^{s-\varepsilon}(\R^n\times (-1,1))} + C\left(||fd^s||_{L^\infty(\Omega\cap Q_1)} + ||u||_{L^\infty(\R^n\times (-1,1))}\right).$$
	The constant $C$ depends only on $n,s,\varepsilon,\delta,G_0$, and ellipticity constants.
\end{lemma}

\begin{proof}
	We can assume that $||fd^s||_{L^\infty(\Omega\cap Q_1)} + ||u||_{L^\infty(\R^n\times (-1,1))}\leq 1$, since otherwise we would divide the equation with a suitable constant. We argue by contradiction. Assume there exists $\delta>0$, so that for every $k\in\N$ there exist domains $\Omega_k$ that are $C^1_p$ in $Q_1$, $u_k,f_k$ and $L_k$ operators satisfying \eqref{eq:operatorForm} so that the suitable equation holds, but 
	\begin{equation}\label{CA}
	\left[ u_k\right]_{C_p^{s-\varepsilon}(\Omega_k\cap Q_{1/2})}> \delta\left[ u_k\right]_{C_p^{s-\varepsilon}(\R^n\times (-1,1))} + k.
	\end{equation}
	Pick $(x_k,t_k),(y_k,s_k)\in \Omega_k\cap Q_{1/2}$ so that
	$$\frac{1}{2}\left[ u_k\right]_{C_p^{s-\varepsilon}(\Omega_k\cap Q_{1/2})}\leq \frac{|u_k(x_k,t_k)-u_k(y_k,s_k)|}{|x_k-y_k|^{s-\varepsilon} + |t_k-s_k|^{\frac{s-\varepsilon}{2s}}}.$$
	We define $\rho_k:= |x_k-y_k| + |t_k-s_k|^{\frac{1}{2s}}$.
	Using \eqref{CA} we see that 
	$$\rho_k^{s-\varepsilon}\left[ u_k\right]_{C_p^{s-\varepsilon}(\Omega_k\cap Q_{1/2})}\leq C |u(x_k,t_k)-u(y_k,s_k)|\leq 2C\leq \frac{C}{k}\left[ u_k\right]_{C_p^{s-\varepsilon}(\Omega_k\cap Q_{1/2})},$$
	which yields that $\rho_k\to0$ as $k\to \infty.$
	
	Now we have the following dichotomy
	\begin{align*}
	&\text{Case }1) \limsup_{k\to\infty} \frac{d_x(x_k,t_k)}{\rho_k} = \infty,\\
	&\text{Case }2) \limsup_{k\to\infty} \frac{d_x(x_k,t_k)}{\rho_k} =:\gamma < \infty.
	\end{align*}
	Let us first treat the case $1)$. We define 
	$$v_k(x,t) = \frac{1}{\left[ u_k\right]_{C_p^{s-\varepsilon}(\R^n\times (-1,1))}\rho_k^{s-\varepsilon}}\left(u_k(x_k+\rho_kx,t_k+\rho_k^{2s}t)-u_k(x_k,t_k)\right).$$ 
	We have $v_k(0,0) = 0$, $\left[ v_k\right]_{C^{s-\varepsilon}(A)}\leq 1$ for any compact set $A\subset R^{n+1}$, provided that $k$ is big enough, and moreover 
	\begin{align*}
	||v_k||_{L^\infty(Q_1)} &\geq |v_k(\rho_k^{-1}(y_k-x_k), \rho_k^{-2s}(s_k-t_k))|\\
	&=\frac{1}{\left[ u_k\right]_{C_p^{s-\varepsilon}(\R^n\times (-1,1))}\rho_k^{s-\varepsilon}} |u_k(y_k,s_k)-u_k(x_k,t_k)|\\
	&\geq \frac{1}{C}\frac{\left[ u_k\right]_{C_p^{s-\varepsilon}(\Omega\cap Q_{1/2})}}{\left[ u_k\right]_{C_p^{s-\varepsilon}(\R^n\times (-1,1))}} \geq \frac{\delta}{C}
	\end{align*}
	where in the last step we used \eqref{CA}. 
	Moreover $v_k$ satisfy
	\begin{equation}\label{rescaledEquation}
	(\partial_t + L_{k})v_k = \frac{\rho_k^{s+\varepsilon}}{\left[ u_k\right]_{C_p^{s-\varepsilon}(\R^n\times (-1,1))}}f_{k,\rho_k}  \quad \text{in }U_k\cap Q_{1/\rho_k}, 
	\end{equation}
	where $U_k=\{(x,t);\text{ }(x_k+\rho_kx,t_k+\rho_k^{2s}t)\in \Omega_k \}$ and  $f_{k,\rho_k}(x,t) = f_k(x_k+\rho_kx,t_k+\rho_k^{2s}t)$.
	Note that $U_k$ converge to $\R^{n+1}$, while
	$$ \left|\frac{\rho_k^{s+\varepsilon}}{\left[ u_k\right]_{C_p^{s-\varepsilon}(\R^n\times (-1,1))}}f_{k,\rho_k}\right|\leq \frac{\rho_k^{s+\varepsilon}  d_{k,\rho_k}^{-s}}{k}\longrightarrow 0,$$
	locally uniformly in $R^{n+1} $, thanks to \eqref{CA}. Moreover, $L_{k}$ converge (up to a subsequence) to an operator $L_0$ satisfying \eqref{eq:operatorForm}, thanks to equi-continuity and uniform boundedness on $\S^{n-1}$.
	
	Passing to subsequence we get that $v_k$ converge to a function $v$ locally uniformly in $\R^{n+1}$. The function $v$ satisfies $v(0) = 0$, $\left[ v\right]_{C^{s-\varepsilon}(\R^{n+1})}\leq 1$, which implies 
	$$||v||_{L^\infty(Q_R)}\leq R^s,\quad\quad R>1.$$
	Thanks to \cite[Lemma 3.1]{FR17}, the function $v$ solves\footnote{Since the kernels $K_k$ of operators $L_k$  are uniformly bounded, a subsequence converges weekly to some measure satisfying the same ellipticity conditions. For more details see \cite[Lemma 3.1]{FR17}.}
	$$(\partial_t + L_0)v=0\quad \text{in }\R^{n+1},$$
	and hence by \cite[Theorem 2.1]{FR17} $v$ is a constant function.
	Hence $v(0,0)=0$ and $||v||_{L^\infty(Q_1)}>0$ contradict each other.
	
	In the case $2)$ we proceed similarly. We choose $(z_k,t_k)\in\partial\Omega\cap Q_{1/2}$ so that $d_x(x_k,t_k)=|z_k-x_k|.$ Then we define
	$$v_k(x,t) = \frac{1}{\left[ u_k\right]_{C_p^{s-\varepsilon}(\R^n\times (-1,1))}\rho_k^{s-\varepsilon}} u_k(z_k+\rho_kx,t_k+\rho_k^{2s}t).$$
	We have $v_k(0,0) = 0,$ $\left[ v_k\right]_{C_p^{s-\varepsilon}(K)} \leq 1$ for every compact set $K\subset\R^{n+1}$ and $k$ big enough, but moreover
	
	$$  |v_k(\xi_k,0) - v_k(\eta_k,\tau_k)|  = \frac{|u_k(x_k,t_k)-u_k(y_k,s_k)|}{\left[ u_k\right]_{C_p^{s-\varepsilon}(\R^n\times (-1,1))} \rho_k^{s-\varepsilon}} \geq \frac{1}{C}\frac{\left[ u_k\right]_{C_p^{s-\varepsilon}(\Omega_k\cap Q_{1/2})}}{\left[ u_k\right]_{C_p^{s-\varepsilon}(\R^n\times (-1,1))}}\geq c_0\delta>0,$$
	thanks to \eqref{CA}, where $\xi_k = \rho_k^{-1}(x_k-z_k)$, $(\eta_k,\tau_k) = (\rho_k^{-1}(y_k-z_k), \rho_k^{-2s}(s_k-t_k))$. All points $(\xi_k,0),(\eta_k,\tau_k)$ are contained in $Q_{\gamma+1}.$
	
	%	$$\left[ v_k\right]_{C_p^{s-\varepsilon}(Q_{2\gamma})}\geq \frac{|u_k(x_k,t_k)-u_k(y_k,s_k)|}{\left[ u_k\right]_{C_p^{s-\varepsilon}(\R^n\times (-1,1))} \rho_k^{s-\varepsilon}} \geq \frac{1}{C}\frac{\left[ u_k\right]_{C_p^{s-\varepsilon}(\Omega_k\cap Q_{1/2})}}{\left[ u_k\right]_{C_p^{s-\varepsilon}(\R^n\times (-1,1))}}\geq c_0\delta>0,$$
	%	thanks to \eqref{CA}. 
	In this case $v_k$ solves 
	$$\left\lbrace
	\begin{array}{rcll}
	(\partial_t  + L_{k})v_k &=& \frac{\rho_k^{s+\varepsilon}}{\left[ u_k\right]_{C_p^{s-\varepsilon}(\R^n\times (-1,1))}}f_{k,\rho_k}  &\text{in }U_k\cap Q_{1/\rho_k} \\
	v_k&=&0&\text{in }U_k^c\cap Q_{1/\rho_k}.
	\end{array}\right.
	$$
	Due to \eqref{CA}, it holds $ \left|\frac{\rho_k^{s+\varepsilon}}{\left[ u_k\right]_{C_p^{s-\varepsilon}(\R^n\times (-1,1))}}f_{k,\rho_k}\right|\leq \frac{\rho_k^{s+\varepsilon}d^{-s}_{k,\rho_k}}{k}.$ 
	Choosing a suitable subsequence $\{k_l\}_{l\in\N}$, so that $v:=\lim\limits_{l\to\infty}v_{k_l},$ $\xi = \lim\limits_{l\to\infty} \xi_{k_l},$ $(\eta,\tau)=\lim\limits_{l\to\infty} (\eta_{k_l},\tau_{k_l})$ exist, we deduce from \cite[Lemma 3.1]{FR17} that
	$$\left\lbrace
	\begin{array}{rcll}
	(\partial_t  + L_{0})v &=& 0  &\text{in }\{x_n>0\} \\
	v&=&0&\text{in }\{x_n\leq 0\},
	\end{array}\right.
	$$
	for some homogeneous operator $L_0$ satisfying \eqref{eq:operatorForm}.
	Moreover $v$ inherits the following properties: $v(0,0)=0,$ $\left[ v\right]_{C_p^{s-\varepsilon}(\R^{n+1})} \leq 1$ and $|v(\xi,0)-v(\eta,\tau)| \geq  c_0\delta>0.$ It follows from \cite[Theorem 2.1]{FR17}, that $v$ is a constant function, and hence $v=0$. But this contradicts the fact that $v$ has different values at $(\xi,0)$ and $(\eta,\tau)$.
\end{proof}

	We show next how to conclude the a priori boundary regularity estimate using this lemma.

\begin{proof}[Proof of Proposition \ref{prop:holderBoundaryEstimate}]
	We choose a smooth cut off function $\eta\in C^{\infty}_c(Q_1)$, so that $\eta\equiv 1 $ in $Q_{3/4}$. We apply Lemma \ref{lem:fakeBoundaryReg} on the truncated function $\eta u$ to obtain that for every $\delta>0$ there is $C>0$ so that
	$$\left[ u\right]_{C_p^{\gamma}(Q_{1/4})}  \leq \delta \left[ \eta u\right]_{C_p^{\gamma}(\R^n\times (-1,1))} + C\left(||gd^{s}||_{L^\infty(\Omega\cap Q_{1/2})} + ||\eta u||_{L^\infty(\R^n\times(1-,1))} \right),$$
	where we denoted $g=(\partial_t  + L)(\eta u)$. Since $\eta\equiv 1$ in $Q_{1/2}$ we have $g=f- Lu(1-\eta)$ in $\Omega\cap Q_{1/2}.$ We can furthermore estimate 
	$$\left[ \eta u\right]_{C_p^{\gamma}(\R^n\times (-1,1))}\leq \left[  u\right]_{C_p^{\gamma}(Q_1)} + \left[ (1-\eta) u\right]_{C_p^{\gamma}(Q_1)}\leq C \left[  u\right]_{C_p^{\gamma}(Q_1)}, $$
	and 
	\begin{align*}
	\left|L(u(1-\eta))(x,t)\right| &\leq \int_{B_{3/4}^c} \left|u(y,t)(1-\eta)(y,t)K(x,y,t)\right| dy \\
	&\leq \Lambda \sup_{R>1}R^{\varepsilon-2s}||u||_{L^\infty(B_R\times (-1,1))}\int_{B_{3/4}^c}|y|^{2s-\varepsilon}|x-y|^{-n-2s}dy
	\\&\leq C\sup_{R>1}R^{\varepsilon-2s}||u||_{L^\infty(B_R\times (-1,1))}.
	\end{align*}
	This gives 
	$$\left[ u\right]_{C_p^{\gamma}(Q_{1/4})}  \leq \delta \left[ u\right]_{C_p^{\gamma}(Q_1)} + C\left(||fd^{s}||_{L^\infty(\Omega\cap Q_{1})} + \sup_{R>1}R^{\varepsilon-2s}|| u||_{L^\infty(B_R\times(1-,1))} \right),$$
	which by \cite[Lemma 2.23]{FR20} and the covering argument proves the result.
\end{proof}

	In the following result we establish a version of the estimate that is used later on.

	\begin{corollary}\label{cor:advancedBoundaryReg}
		Let $s\in(\frac{1}{2},1)$. Let $\Omega\subset \R^{n+1}$ be $C^{1}_p$ in $Q_1$. Let $\gamma\in(0,s)$ and $\varepsilon>0$, $\gamma_2\in(0,1)$. Let $L$ be an operator of the form \eqref{eq:operatorForm}. Assume $u\in C_p^{\gamma}(Q_1)$ is a solution of
		$$\left\lbrace\begin{array}{rcll}
		(\partial_t  + L)u &=& f_1+f_2& \text{in }\Omega\cap Q_1\\
		u&=&0&\text{in }\Omega^c\cap Q_1,
		\end{array}\right. $$
		with $f_1d^{s}\in L^\infty(\Omega\cap Q_1)$,  and $f_2\in C^{\gamma_2}_p(\Omega\cap Q_1).$
		Then 
		\begin{align*}
			\left[ u\right]_{C_p^{\gamma}( Q_{1/2})}\leq & C\left(||f_1d^{s}||_{L^\infty(\Omega\cap Q_1)} + \left[ f_2\right]_{C^{\gamma_2}_p} + \sup_{R>1}R^{-2s+\varepsilon}||u||_{L^\infty(B_R\times(-1,1))}\right).
		\end{align*}
		The constant $C$ depends only on $n,s,\gamma,\varepsilon,G_0$,  and ellipticity constants.
	\end{corollary}
	\begin{proof}
		To pass from $L^\infty$ norm to the H\"older seminorm in the right hand side, we  find function $v$ solving
		$$\left\lbrace\begin{array}{rcll}
		(\partial_t  + L)v &=& f_2(0,0)& \text{in }\Omega\cap Q_1\\
		v&=&0&\text{in }(\Omega\cap Q_1)^c.
		\end{array}\right. $$
		Then we have $|f_2(0,0)|\leq  C||v||_{L^\infty(\Omega\cap Q_1)}$ and by comparison principle 
		\begin{align*}
			||u-v||_{L^\infty (\Omega\cap Q_1)}&\leq C(||f_1d^s|| + ||f_2-f_2(0,0)||+||u||_{L^\infty(Q_1^c)})\\
			&\leq C(||f_1d^s|| +  \left[ f_2\right]_{C^{\gamma_2}_p} + ||u||_{L^\infty(Q_1^c)}) .
		\end{align*} 
		Furthermore it holds
		\begin{align*}
			||(f_1+f_2)d^s||_{L^\infty(Q_1\cap\Omega)}&\leq ||f_1d^s||_{L^\infty} + ||f_2-f_2(0,0)|| + |f_2(0,0)|\\
			&\leq  ||f_1d^s||_{L^\infty} + \left[ f_2\right]_{C^{\gamma_2}_p} + C||v||_{L^\infty(Q_1)}\\
			 & \leq C\left( ||f_1d^s||_{L^\infty} + \left[ f_2\right]_{C^{\gamma_2}_p} + ||u||_{L^\infty(Q_1)} + ||u-v||_{L^\infty(Q_1)}   \right)\\
			 &\leq C\left( ||f_1d^s||_{L^\infty} + \left[ f_2\right]_{C^{\gamma_2}_p}+ ||u||_{L^\infty(R^n\times(-1,1))}\right),
		\end{align*}
		and hence by Proposition \ref{prop:holderBoundaryEstimate} we conclude
		\begin{align*}
		\left[ u\right]_{C_p^{\gamma}( Q_{1/2})}\leq & C\left(||f_1d^{s}||_{L^\infty(\Omega\cap Q_1)} + \left[ f_2\right]_{C^{\gamma_2}_p}+ ||u||_{L^\infty(R^n\times(-1,1))} \right).
		\end{align*}
		
		It remains to replace the $L^\infty$ norm of $u$ at infinity with the term as in the statement, which is done with the same cut-off procedure as in the proof of Proposition \ref{prop:holderBoundaryEstimate}.
%		We establish it with choosing a smooth cut-off $\chi\in C^\infty_c(B_1)$ so that $\chi\equiv 1$ in $B_{5/6}$. Then the function $\bar{u} = u\chi$  is in $C^\gamma_p(\R^n\times(-1,1))$ and solves 
%		$$\left\lbrace\begin{array}{rcll}
%		(\partial_t  + L)\bar{u} &=& f_1+f_2+f_3+\bar{f}& \text{in }\Omega\cap Q_1\\
%		\bar{u}&=&0&\text{in }(\Omega\cap Q_1)^c,
%		\end{array}\right. $$
%		where $\bar{f}= -L(u(1-\chi))$. We can estimate		
%		\begin{align*}
%		|\bar{f}(x,t)| =&  \left| \int_{B_{4/5}^c(-x)}u(1-\chi)(x+y,t)K(y)dy \right|\\
%		\leq& \int_{B_{1/20}^c}|u(x+y,t)|K(y)dy\\
%		\leq & CC_0
%		\end{align*}
%		where $C_0$ stands for $\sup_{R>1} R^{-2s+\varepsilon}  || u_1||_{L^{\infty}(B_R\times(-1,1))}.$ The claim now follows from the already proven case above.
	\end{proof}

	\subsection{Boundary Harnack in $C^1_p$ domains}
	
	We now turn our attention to the derivation the boundary Harnack inequalities. To prove that the quotient of two solutions is regular of some order $\alpha$ up to the boundary, it is heuristically enough to prove that the solution in the numerator can be approximated by the solution from the denominator multiplied with a polynomial of order $\alpha$ up to order $\alpha + s$. The regularity of the quotient can then be deduced from interior regularity estimates. We start with establishing such expansions of order $2s$. 
	
	\begin{proposition}\label{boundaryHarnackExpansion}
		Let $s\in(\frac{1}{2},1)$. Let $\Omega\subset \R^{n+1}$ be $C^{1}_p$ in $Q_1$. Let $\varepsilon>0$ and $\rho \in\left(0,s \right) $. Let $L$ be an operator of the form \eqref{eq:operatorForm}. Assume $u_i\in C_p^{\gamma}(Q_1)$, $i\in\{1,2\}$, $\gamma>0$, solve
		$$\left\lbrace\begin{array}{rcll}
		(\partial_t +  L)u_i &=& f_i& \text{in }\Omega\cap Q_1\\
		u_i&=&0&\text{in }\Omega^c\cap Q_1,
		\end{array}\right. $$
		with  $||d^\rho f_i||_{L^\infty(\Omega\cap Q_1)}\leq 1$, $||u_i||_{L^\infty(B_R\times(-1,1))}\leq R^{2s-\rho }$ for $R>1$. Assume also that $u_2\geq c_0 d^s$, for some $c_0>0$.
		
		Then for every $(z,t_0)\in \partial \Omega \cap Q_{1/2}$ there exists a constant $q_{(z,t_0)}$, so that 
		$$|u_1(x,t)-q_{(z,t_0)}u_2(x,t)|\leq C \left(|x-z|+|t-t_0|^{\frac{1}{2s}}\right)^{2s-\rho }.$$
		The constant $C$ depends only on $n,s,\rho  ,\varepsilon,c_0,G_0$, and ellipticity constants.
		
		Moreover for every $(x_0,t_0)\in \Omega\cap Q_{1/2}$, such that $d_x(x_0) = |x_0-z| = c_\Omega r,$ we have 
		\begin{equation*}
		\left[ u_1 - q_{(z,t_0)} u_2 \right] _{C^{2s-\rho }_p(Q_r(x_0,t_0))}\leq C.
		\end{equation*}
	\end{proposition}
	\begin{proof}
		Thanks to the translation and rescaling invariance of the statement, we only need to prove the case when $(z,t_0)=(0,0)$. We prove the claim by contradiction. Assume that for every $k\in\N$, there exist domains $\Omega_k$ that are $C^{1}_p$ in $Q_1$, $u_{i,k},f_{i,k}$, with $||d^\rho f_i||_{L^\infty(\Omega\cap Q_1)},\sup_{R>1}R^{\varepsilon-2s}||u_i||_{L^\infty(B_R\times(-1,1))}\leq1,$ and $L_k$ operators satisfying \eqref{eq:operatorForm}  so that the suitable equations hold, but for any choice of $q_k$ we have
		$$\sup_{r>0} \frac{1}{r^{2s-\rho }}||u_{1,k} - q_ku_{2,k}||_{L^\infty(Q_r)}>k.$$
		Next we define 
		$$q_{k,r} = \frac{\int_{Q_r}u_{1,k}u_{2,k}}{\int_{Q_r}u_{2,k}^2},$$
		so that $\int_{Q_r}(u_{1,k}-q_{k,r}u_{2,k})u_{2,k} = 0,$
		and set
		$$\theta(r) = \sup_{k} \sup_{r'>r} \frac{1}{r'^{2s-\rho }}||u_{1,k}-q_{k,r'}u_{2,k}||_{L^\infty(Q_{r'})}.$$
		Notice that $\theta$ is monotone decreasing and thanks to \cite[Lemma B.7]{K21b} and the contradiction assumption it holds $\lim_{r\downarrow0}\theta(r)=\infty.$
		Now choose sequences $r_m$ and $k_m$, so that 
		$$\frac{m}{4}\leq\frac{1}{2}\theta(r_m)\leq \frac{1}{r_m^{2s-\rho 
		}}||u_{1,k_m}-q_{k_m,r_m}u_{2,k_m}||_{L^\infty(Q_{r_m})}\leq \theta(r_m),$$
		and define
		$$v_m(x,t) = \frac{1}{\theta(r_m)r_m^{2s-\rho }} (u_{1,k_m}-q_{k_m,r_m}u_{2,k_m})(r_mx,r_m^{2s}t).$$
		Next we estimate
		\begin{align*}
			|q_{k,r}-q_{k,2r}|&\leq Cr^{-s} ||q_{k,r}u_{2,k}-q_{k,2r}u_{2,k}||_{L^\infty(Q_r\cap \{d_k>r/2\})}\\
			&\leq Cr^{-s}(||u_{1,k}-q_{k,r}u_{2,k}||_{L^\infty(Q_r)}+||u_{1,k}-q_{k,2r}u_{2,k}||_{L^\infty(Q_{2r})}) \\
			&\leq  C\theta(r)r^{s-\rho }.
		\end{align*}
		This implies in the same way as in \cite[Proposition 4.4]{AR20}, that $\frac{q_{k,r}}{\theta(r)}\to 0,$ as $r\downarrow0,$ and $||v_m||_{L^\infty(Q_R)}\leq CR^{2s-\rho }.$ Moreover by the definition of $\theta$ it holds $||v_m||_{L^\infty(Q_1)}\geq \frac{1}{2}.$
		Let us turn to the equation that $v_m$ satisfies
		$$\left\lbrace
		\begin{array}{rcll}
		(\partial_t + L_{k_m})v_m& =& \frac{r_m^\rho }{\theta(r_m)}(f_{1,k_m}-q_{k_m,r_m}f_{2,k_m})&\text{in } U_m\cap Q_{1/r_m}\\
		u&=&0&\text{in }U_m^c\cap Q_{1/r_m},
		\end{array}
		\right.$$
		with notation as in \eqref{rescaledEquation}. Rescaled Proposition \ref{prop:holderBoundaryEstimate} say, that for every $M>1$ we get that
		$$\left[ v_m\right]_{C_p^{\gamma}(Q_{M/2})}\leq C(M),$$
		provided that $m$ is big enough. This along with the convergence result \cite[Lemma 3.1]{FR17} allows us to pass to the limit (up to a subsequence) to get that $v:=\lim_{m\to\infty}v_m$ solves
		$$\left\lbrace
		\begin{array}{rcll}
		(\partial_t + L_{0})v& =& 0&\text{in } \{x_n>0\}\\
		u&=&0&\text{in }\{x_n\leq 0\},
		\end{array}
		\right.$$
		for some operator $L_0$ satisfying \eqref{eq:operatorForm}. Moreover it holds $||v||_{L^\infty(Q_1)}\geq \frac{1}{2}$, $||v||_{L^\infty(Q_R)}\leq CR^{2s-\rho }$ and $\int_{Q_1}v\cdot (x_n)_+^s = 0$. Then the Liouville theorem from \cite[Theorem 4.11]{FR17} implies that $v (x)= q (x_n)_+^s$, which gives rise to the contradiction.
		
		To prove the moreover part, we first notice that $q_{(z,t_0)}$ are bounded independently of the point $(z,t_0)$, thanks to \cite[Lemma B.5]{K21b}.
		Then we apply interior estimates (Lemma~\ref{lem:interiorRegularityWithGrowth}) on rescaled function $u_1 - q_{(z,t_0)}u_2$ together with the above established estimate with \cite[Lemma A.3]{K21a}, to get the wanted estimate.
	\end{proof} 

	Note that the argument works only for orders smaller than $2s$, since the growth of the blow-up function at infinity inherits this order, which is the largest admissible growth when dealing with non-local operators of order $2s$.
	
	As mentioned before, the established expansion implies the boundary regularity estimate for the quotient of two solutions, for orders smaller than $s$.

	\begin{corollary}\label{cor:basicQuotientEstimate}
	Let $s\in(\frac{1}{2},1)$. Let $\Omega\subset \R^{n+1}$ be $C^{1}_p$ in $Q_1$. Let  $\rho \in(0,s)$. Let $L$ be an operator of the form \eqref{eq:operatorForm}. Assume $u_i\in C_p^{\gamma}(Q_1)$, $i\in\{1,2\}$, $\gamma>0$, solve
	$$\left\lbrace\begin{array}{rcll}
	(\partial_t  + L)u_i &=& f_i& \text{in }\Omega\cap Q_1\\
	u_i&=&0&\text{in }\Omega^c\cap Q_1,
	\end{array}\right. $$
	with  $||d^\rho f_2||_{L^\infty(\Omega\cap Q_1)}\leq 1$ and $||u_2||_{L^\infty(B_R\times(-1,1))}\leq R^{2s-\rho }$ for $R>1$. Assume also that $u_2\geq c_0 d^s$, for some $c_0>0$.
	
	Then 
	$$\left|\left| \frac{u_1}{u_2}\right|\right| _{C_p^{s-\rho }(\overline{\Omega}\cap Q_{1/2})}\leq C\left(||d^\rho f_1||_{L^\infty(\Omega\cap Q_1)} +\sup_{R>1} R^{-2s+\rho } ||u_1||_{L^\infty(B_R\times(-1,1))}\right),$$
	with $C>0$ depending only on $n,s,\rho ,c_0,G_0$,  and ellipticity constants.
\end{corollary} 

\begin{proof}
	Dividing $u_1$ with $||d^\rho f_1||_{L^\infty(\Omega\cap Q_1)} + \sup_{R>1} R^{-2s+\rho } ||u_1||_{L^\infty(B_R\times(-1,1))}$ and thanks to \cite[Lemma B.2]{K21b} it suffices to prove $$\left[ \frac{u_1}{u_2}\right] _{C_p^{s-\rho }(Q_r(x_0,t_0))}\leq C,$$
	independently of $(x_0,t_0)$ and $r$ whenever $Q_{2r}(x_0,t_0)\subset \Omega\cap Q_{1/2}$ and $d_{x}(x_0,t_0)\leq C_1 r$, with $C_1$ depending only on $\Omega $. Denote $(z,t_0)\in\partial\Omega$ the closest point to $(x_0,t_0),$ and $C_1r = |x_0-z|$. Proposition \ref{boundaryHarnackExpansion} gives 
	$$\left[ u_1-qu_2\right]_{C^{2s-\rho }_p(Q_r(x_0,t_0))}\leq C, \quad\quad ||u_1-qu_2||_{L^{\infty}(Q_r(z,t_0))}\leq Cr^{2s-\rho },$$
	for a suitable constant $q$. We now apply Lemma \ref{lem:divisionLemma} to get
	$$\left[ \frac{u_1}{u_2}\right]_{C_p^{s-\rho }(Q_r(x_0,t_0))}\leq C,$$
	where the interior regularity for $u_2$ is provided by Lemma \ref{lem:interiorRegularityWithGrowth}. The claim is proven.
\end{proof}

We prove next an estimate for the regularity of the quotient of two solutions in the form needed later on.

\begin{corollary}\label{cor:advancedQuotientEstimate}
	Let $s\in(\frac{1}{2},1)$. Let $\Omega\subset \R^{n+1}$ be $C^{1}_p$ in $Q_1$. Let $\varepsilon,\varepsilon'>0$ and $\gamma\in (0,1)$. Let $L$ be an operator of the form \eqref{eq:operatorForm}. Assume $u_i\in C_p^{\varepsilon'}(\R^n\times (-1,1))$, $i\in\{1,2\}$, solve
	$$\left\lbrace\begin{array}{rcll}
	(\partial_t + L)u_i &=& f_i& \text{in }\Omega\cap Q_1\\
	u_i&=&0&\text{in }\Omega^c\cap Q_1,
	\end{array}\right. $$
	with $||f_2||_{L^\infty(\Omega\cap Q_1)}\leq 1$, $||u_2||_{L^\infty(B_R\times(-1,1))}\leq R^{2s-\varepsilon}$, for $R>1$. Assume also that $u_2\geq c_0 d^s$, for some $c_0>0$.
	
	Then 
	$$\left|\left| \frac{u_1}{u_2}\right|\right| _{C_p^{s-\varepsilon}(\overline{\Omega}\cap Q_{1/2})}\leq C\left(\left[ f_1\right] _{C^\gamma_p(\Omega\cap Q_1)} + \left|\left|u_1\right|\right|_{L^\infty(\R^{n}\times(-1,1))}\right),$$
	with $C>0$ depending only on $n,s,\varepsilon,c_0,G_0$, and ellipticity constants.
	
	Moreover
	$$\left|\left| \frac{u_1}{u_2}\right|\right| _{C_p^{s-\varepsilon}(\overline{\Omega}\cap Q_{1/2})}\leq C\left(\left[ f_1\right] _{C^\gamma_p(\Omega\cap Q_1)} + \sup_{R>1} R^{-2s+\varepsilon}  \left[ u_1\right]_{C^{\gamma}_p(B_R\times(-1,1))} + ||u_1||_{L^\infty(Q_1)}\right),$$
	and if the kernel $K$ of the operator $L$ satisfies $\left[ K\right]_{C^\gamma(B_r^c)}\leq C r^{-n-2s-\gamma},$ we have
	\begin{align*}
	\left|\left| \frac{u_1}{u_2}\right|\right| _{C_p^{s-\varepsilon}(\overline{\Omega}\cap Q_{1/2})}\leq& C\left(\left[ f_1\right] _{C^\gamma_p(\Omega\cap Q_1)} + \sup_{R>1} R^{-2s+\varepsilon}  \left[ u_1\right]_{C^{\frac{\gamma}{2s}}_t(B_R\times(-1,1))}\right. \\
	&\quad\quad+ \left.\sup_{R>1} R^{-2s-\gamma-\varepsilon}||u_1||_{L^\infty(B_R\times(-1,1))}\right).
	\end{align*}
\end{corollary}
\begin{proof}
	We split $u_1 = u+v$, where $(\partial_t + L)v = f_1(0,0)$ in $\Omega$, and $v=0$ in $\Omega^c$. Then we have that $|f_1(0,0)|\leq C ||v||_{L^\infty(\Omega)}$ and $||u||_{L^\infty(\Omega\cap Q_1)}\leq C(||f_1-f_1(0,0)||_{L^\infty(\Omega\cap Q_1)})\leq C\left[ f_1\right] _{C^\gamma_p(\Omega\cap Q_1)}.$ Hence 
	\begin{align*}
	||f_1||_{L^\infty(\Omega\cap Q_1)} & \leq |f_1(0,0)| + C\left[ f_1\right] _{C^\gamma_p(\Omega\cap Q_1)}\leq C\left(||v||_{L^\infty(\Omega\cap Q_1)} + \left[ f_1\right] _{C^\gamma_p(\Omega\cap Q_1)}\right)\\
	&\leq C \left(||u_1||_{L^\infty(\Omega\cap Q_1)} + ||u||_{L^\infty(\Omega\cap Q_1)} + \left[ f_1\right] _{C^\gamma_p(\Omega\cap Q_1)}\right)\\
	&\leq  C \left(||u_1||_{L^\infty(\Omega\cap Q_1)} + \left[ f_1\right] _{C^\gamma_p(\Omega\cap Q_1)}\right).
	\end{align*}
	The first estimate now follows from Corollary \ref{cor:basicQuotientEstimate}.
	
	To prove the moreover case, we use the already proven result on the function $\bar{u} = u_1\chi$ instead of $u_1$, where $\chi \in C^\infty_c(B_1)$, with $\chi\equiv 1$ in $B_{4/5}$.
	Then $\bar{u}$ solves $(\partial_t+L) \bar{u} = (\partial_t+L) u_1 + L(u_1(\chi-1)) = f_1 + \bar{f} $ in $Q_{3/4}\cap\Omega$. We estimate $\left[ \bar{f}\right] _{C^\gamma_p(\Omega\cap Q_{3/4})}$ as follows
	\begin{align*}
	|\bar{f}(x,t)-\bar{f}(x',t')| =&  \left| \int_{B_{4/5}^c(-x)}u_1(1-\chi)(x+y,t)K(y)dy\right.\\
	&\left.-\int_{B_{4/5}^c(-x')}u_1(1-\chi)(x'+y,t')K(y)dy \right|\\
	\leq& \int_{B_{1/20}^c}|u_1(x+y,t)-u_1(x'+y,t')|K(y)dy\\
	&+ ||u_1||_{L^\infty(Q_1)}\int_{B_{1/20}^c}|\chi(x+y)-\chi(x'+y)|K(y)dy\\
	\leq&  (|x-x'|^\gamma+|t-t'|^{\frac{\gamma}{2s}})\left(\int_{B_{1/20}^c}\left[u_1 \right]_{C^\gamma_p(B_{|y|+1})}K(y)dy + C_\chi||u_1||_{L^\infty(Q_1)}\right)\\
	\leq&  (|x-x'|^\gamma+|t-t'|^{\frac{\gamma}{2s}})\left(\Lambda C_0
	\int_{B_{1/2}^c}\frac{(|y|+1)^{2s-\varepsilon}}{|y|^{n+2s}}dy+C||u_1||_{L^\infty(Q_1)}\right)\\
	\leq& (|x-x'|^\gamma+|t-t'|^{\frac{\gamma}{2s}})\left(CC_0+C||u_1||_{L^\infty(Q_1)}\right),
	\end{align*}
	where $C_0$ stands for $\sup_{R>1} R^{-2s+\varepsilon}  \left[ u_1\right]_{C^{\gamma}_p(B_R\times(-1,1))}.$
	Applying the already proven inequality for $Q_{3/4}$ and $Q_{1/2}$ (which follows from the covering argument) gives 
	$$\left[ \frac{u_1}{u_2}\right] _{C_p^{s-2\varepsilon}(\overline{\Omega}\cap Q_{1/2})}\leq C\left(\left[ f_1\right] _{C^\gamma_p(\Omega\cap Q_1)} +\sup_{R>1} R^{-2s+\varepsilon}  \left[ u_1\right]_{C^{\gamma}_p(B_R\times(-1,1))}+ \left|\left|u_1\right|\right|_{L^\infty(Q_1)}\right).$$
	To pass from the growth control with the parabolic H\"older seminorms to seminorms in time with $L^\infty$ growth, we proceed in the same way as in the moreover case of Lemma \ref{lem:interiorRegularityWithGrowth}.
	The claim is proven.
\end{proof}

In the following proposition we establish the expansion for two solutions for orders between $2s$ and $1+s$. In order to exceed order $2s$, we work with H\"older seminorms of the quotient directly. We follow the idea of \cite[Proposition 3.3]{RV18}. 

\begin{proposition}\label{prop:decay2functions}
	Let $s\in(\frac{1}{2},1)$. Let $\Omega\subset \R^{n+1}$ be $C^{1}_p$ in $Q_1$. Let $\varepsilon>0$ and $\gamma>0$. Let $L$ be an operator of the form \eqref{eq:operatorForm}, with kernel $K\in C^{1-s}(\S^{n-1})$ satisfying $\left[K\right]_{C^{1-s}(\S^{n-1})}\leq M_0$, for some $M_0>0$. Assume $u_i\in C_p^{\gamma}(Q_1)$, $i\in\{1,2\}$, solve
	$$\left\lbrace\begin{array}{rcll}
	(\partial_t + L)u_i &=& f_i& \text{in }\Omega\cap Q_1\\
	u_i&=&0&\text{in }\Omega^c\cap Q_1,
	\end{array}\right. $$
	with $\left[ f_i\right]_{C^{1-s}_p(\Omega\cap Q_1)}\leq 1$, $\left[ u_i\right] _{C^{\frac{1-s}{2s}}_t(B_R\times(-1,1))}\leq R^{2s-1}$, and $||u_2||_{L^\infty(B_R\times(-1,1))}\leq R^s$ and $||u_1||_{L^\infty(B_R\times(-1,1))}\leq R^{3s-1-\varepsilon}$. Assume also that $u_2\geq c_0 d^s$, for some $c_0>0$.
	
	Then for every $(z,t)\in\partial\Omega\cap Q_{1/2}$ we have
	$$\left[ \frac{u_1}{u_2}-q_r\right] _{C_p^{s-2\varepsilon}(\Omega\cap Q_{r}(z,t))}\leq Cr^{1-s},$$
	where $q_r$ equals $\operatorname{arg}\min\limits_{q\in\R}\int_{Q_r\cap \Omega} (\frac{u_1}{u_2}-q)^2 $ and $C>0$ depends only on $n,s,\varepsilon,c_0,G_0,M_0$ and ellipticity constants.
\end{proposition}
\begin{proof}
	Assume first that $u_1\equiv 0$ in $Q_1^c$. Without loss of generality we can assume that $(z,t)=(0,0)$.  
	We argue by contradiction. Suppose that for each $k\in\N$ there exist $\Omega_k,u_{i,k},f_{i,k},L_k$ as in the statement, so that 
	$$\sup_k\sup_{r>0}r^{s-1}\left[ \frac{u_{1,k}}{u_{2,k}}-q_{r,k}\right] _{C_p^{s-2\varepsilon}(\Omega_k\cap Q_r)} = \infty,$$
	with $q_{r,k} = \operatorname{arg}\min\limits_{q\in\R}\int_{Q_r\cap \Omega_k} (\frac{u_{1,k}}{u_{2,k}}-q)^2 $.
	
	Then we define 
	$$\theta(r) = \sup_k\sup_{\rho>r} \rho^{s-1}\left[ \frac{u_{1,k}}{u_{2,k}}-q_{\rho,k}\right] _{C_p^{s-2\varepsilon}(\Omega_k\cap Q_\rho)},$$
	which is monotone in $r$, finite for $r>0$ by Corollary \ref{cor:advancedQuotientEstimate}, %hint: one can multiply u_1k with a cuttoff \chi = 1 in Q_4/5 compactly supp in Q_1, so that \theta is defined for all r>0. thish shouldnt change the seminorm of f_1 infinitely, since u_1 is C^{s-\vare}_p in the full space, and 1-s is smaller.
	and goes to $\infty$ as $r\downarrow0$, thanks to the contradiction assumption. Hence for every $m\in\N$ we can find $k_m$ and $r_m$, so that 
	$$\frac{m}{4}\leq\frac{\theta(r_m)}{2}\leq r_m^{s-1}  \left[ \frac{u_{1,k_m}}{u_{2,k_m}}-q_{r_m,k_m}\right] _{C_p^{s-2\varepsilon}(\Omega_{k_m\cap Q_{r_m})}}\leq\theta(r_m).$$
	In particular $r_m\to 0$ as $m\to\infty.$
	We define the blow up sequence 
	$$v_m(x,t) = \frac{1}{\theta(r_m)r_m^{1-2\varepsilon}} \left(\frac{u_{1,k_m}}{u_{2,k_m}} - q_{r_m,k_m}\right)(r_mx,r_m^{2s}t).$$
	Notice that by definition of $r_m,k_m$ and $q_{r,m}$ we have
	\begin{equation}\label{eq:contradictionQuantities}
	\left[ v_m\right]_{C_p^{s-2\varepsilon}(\Omega_m\cap Q_1)} \geq \frac{1}{2}\quad\text{and}\quad \int_{\Omega_m\cap Q_1} v_m = 0,
	\end{equation}
	where $\Omega_m= \{(x,t); \text{ } (r_mx,r_m^{2s}t)\in\Omega_{k_m}  \}$.
	Moreover, the way $\theta$ is defined gives a control on the growth of $v_m$ in the following manner
	\begin{align}\label{eq:gorwthControl}
	\begin{split}
	\left[ v_m\right]_{C_p^{s-2\varepsilon}(\Omega_m\cap Q_R)} = & \frac{r_m^{s-2\varepsilon}}{\theta(r_m)r_m^{1-2\varepsilon}}\left[ \frac{u_{1,k_m}}{u_{2,k_m}}- q_{r_m,k_m}\right]_{C_p^{s-2\varepsilon}(\Omega_m\cap Q_{Rr_m})}\\
	\leq &  \frac{1}{\theta(r_m)r_m^{1-s}}\left[ \frac{u_{1,k_m}}{u_{2,k_m}}- q_{Rr_m,k_m}\right]_{C_p^{s-2\varepsilon}(\Omega_m\cap Q_{Rr_m})}\\
	\leq & \frac{\theta(Rr_m)(Rr_m)^{1-s}}{\theta(r_m)r_m^{1-s}}\leq R^{1-s},
	\end{split}
	\end{align}
	for every $R\geq 1.$ 
	In combination with $\int_{\Omega_m\cap Q_1} v_m = 0$ we deduce  $||v_m||_{L^{\infty}(Q_1\cap\Omega_m)}\leq C$ uniformly in $m$. Combining it with the growth of the seminorms, we deduce  $||v_m||_{L^{\infty}(Q_R\cap\Omega_m)}\leq CR^{1-2\varepsilon}.$ Notice that these estimates hold true for general $k,r$, since we only used the definition of $\theta$ and not of $r_m,k_m$. Therefore we have 
	\begin{align*}
	\frac{|q_{r,k}-q_{2r,k}|}{\theta(r)} &=\frac{||q_{r,k}-q_{2r,k}||_{L^\infty(Q_r\cap\Omega_k)}}{\theta(r)}\\
	&\leq \frac{||u_{1,k}/u_{2,k-q_{r,k}}||_{L^\infty(Q_r\cap\Omega_k)}}{\theta(r)} + \frac{||u_{1,k}/u_{2,k-q_{2r,k}}||_{L^\infty(Q_r\cap\Omega_k)}}{\theta(r)} \\
	&\leq r^{1-2\varepsilon}\left|\left|\frac{1}{\theta(r)r^{1-2\varepsilon}}\left(\frac{u_{1,k}}{u_{2,k}} - q_{r,k}\right)\right|\right|_{L^\infty(\Omega_k\cap Q_r)}\\
	&\text{    }+ (2r)^{1-2\varepsilon}\left|\left|\frac{1}{\theta(2r)(2r)^{1-2\varepsilon}}\left(\frac{u_{1,k}}{u_{2,k}} - q_{r,k}\right)\right|\right|_{L^\infty(\Omega_k\cap Q_{2r})}\\
	&\leq Cr^{1-2\varepsilon}.
	\end{align*}
	Hence as in the proof of \cite[Proposition 4.1]{AR20}, we conclude that $\frac{q_{r,k}}{\theta(r)}\to 0 $ as $r\downarrow0$, uniformly in $k$.
	
	We now define $$v_{1,m}(x,t) = \frac{1}{\theta(r_m)r_m^{1+s-2\varepsilon}} \left(u_{1,k_m}- q_{r_m,k_m}u_{2,k_m} \right)(r_mx,r_m^{2s}t),$$ 
	and 
	$$v_{2,m} (x,t) = \frac{1}{r_m^s}u_{2,k_m} (r_mx,r_m^{2s}t),$$
	so that $v_m = \frac{v_{1,m}}{v_{2,m}}.$ Thanks to assumptions on $u_{2,k}$ we have that $\left[ v_{2,m}\right] _{C^s_p(\R^{n}\times(-r_{k_m}^{-2s},r_{k_m}^{-2s}))}\leq 1$, $||(\partial_t + L_{k_m}) v_{2,m}||_{L^\infty(Q_1\cap\Omega_m)}\leq 1$ and $v_{2,m}\geq cc_0 d_m^s$, were $d_m$ denotes the distance function in $\Omega_m$. Therefore we can apply  Corollary \ref{cor:advancedQuotientEstimate} to $u_{1,m}(M\cdot, M^{2s}\cdot)$ and $v_{2,m}(M\cdot, M^{2s}\cdot)$, to get
	\begin{align*}
	\left[ v_m\right]_{C_p^{s-\varepsilon}(\Omega_m\cap Q_{M/2})} \leq C(M) &\left(\left[(\partial_t + L_{k_m})v_{1,m} \right]_{C^{1-s}_p(Q_M\cap\Omega_m)} \right.
	\\&\left.+ \sup_{R>M}R^{\varepsilon-2s}\left[ v_{1,m}\right]_{C^{\frac{1-s}{2s}}_t(B_R\times(-M^{2s},M^{2s}))} \right.\\
	&\left.+ \sup_{R>1}R^{-s-1+\varepsilon}||v_{1,m}||_{L^\infty(B_R\times(-M^{2s},M^{2s}))} \right),
	\end{align*}
	for every $M\in\N$, for $m$  big enough.
	Let us now show that all the quantities in the right-hand side are bounded uniformly in $m$. First, 
	$$(\partial_t + L_{k_m})v_{1,m}(x,t) = \frac{1}{\theta(r_m)r_m^{1-s-2\varepsilon}} (f_{1,k_m} - q_{r_m,k_m} f_{2,k_m})(r_mx,r_m^{2s}t),$$
	and hence  
	\begin{align*}
	\left[(\partial_t + L_{k_m})v_{1,m} \right]_{C^{1-s}_p(Q_M\cap\Omega_m)} &= \frac{r_m^{2\varepsilon	}}{\theta(r_m)}\left[f_{1,k_m} - q_{r_m,k_m} f_{2,k_m}  \right]_{C^{1-s}_p(Q_{Mr_m\cap\Omega_{k_m}})}\\
	&\leq \frac{C(1+q_{r_m,k_m})}{\theta(r_m)}, 
	\end{align*}
	which is bounded by assumption on the seminorms of $f_{i,k}$ and the fact that $\frac{q_{r,k}}{\theta(r)}\to 0$ as $r\downarrow 0$ uniformly in $k$. 
	The growth term we estimate as follows
	\begin{align*}
	\left[ v_{1,m}\right]_{C^{\frac{1-s}{2s}}_t(B_R\times(-M^{2s},M^{2s}))} & \leq \left[ v_m v_{2,m}\right]_{C^{\frac{1-s}{2s}}_t}\\
	&\leq \left[ v_m \right]_{C^{s-2\varepsilon}_p(Q_R)} R^{s-2\varepsilon -1+s}||v_{2,m}||_{L^\infty} + \left[ v_{2,m} \right]_{C^{\frac{1-s}{2s}}_t}||v_{m}||_{L^\infty(Q_R)}\\
	&\leq R^{1-s} R^{2s-1-2\varepsilon} C(M) R^s + C(M) R^{2s-1} R^{1-2\varepsilon} \\
	&\leq C(M)R^{2s-2\varepsilon}, 
	\end{align*}
	where we used the growth control of $v_m$ from \eqref{eq:gorwthControl} and  $\left[ u_{2,k}\right]_{C^{1-s}_p(B_R\times(-1,1))}\leq R^{2s-1},$ which holds by assumption.
	%		Hence 
	%		$$\left[ v_{1,m}\right]_{C^{1-s}_p(B_R\times(-M^{2s},M^{2s}))}\leq C(M) R^{2s-2\varepsilon}.$$
	The last term is estimated similarly
	$$||v_{1,m}||_{L^\infty(B_R\times(-M^{2s},M^{2s}))}\leq  ||v_{m}||_{L^\infty}||v_{2,m}||_{L^\infty}\leq C(M)R^{1-2\varepsilon}R^s.$$
	Putting it all together, we get that $|| v_m||_{L^\infty(\Omega_m\cap Q_{M})}$ and $\left[ v_m\right]_{C_p^{s-\varepsilon}(\Omega_m\cap Q_{M})}$ are uniformly bounded, which implies that $v_m$ converge to some function $v$ in $C^{s-2\varepsilon}_p$ in any compact subset of $\{x_n\geq 0\}.$ Moreover $v_{2,m}$ converges to $c(x_n)_+^s$ in $C^{s-2\varepsilon}_p$ locally in $\R^{n+1}.$ Hence also $v_{1,m}$ converges to $w:=v\cdot(x_n)_+^s$ in $C^{s-2\varepsilon}_p$ locally in $\R^{n+1}.$
	
	We claim that $w$ satisfies the hypothesis of the Liouville theorem \cite[Theorem 2.1]{RV18}. To see this, note that for fixed $h\in\R^n$, $h_n\geq 0$, and $\tau\in\R$ we have
	$$(\partial_t + L_{k_m})(v_{1,m}(x+h,t+\tau) - v_{1,m}(x,t)) = \frac{1}{\theta(r_m)r_m^{1-s}}\left( \hat{f}_{1,k_m} + q_{r_m,k_m}\hat{f}_{2,k_m}\right),$$
	where $\hat{f}_{i,k_m} = f_{i,k_m}(r_m(x+h),r_m^{2s}(t+\tau)) -  f_{i,k_m}(r_mx,r_m^{2s}t)$. But 
	$$|\hat{f}_{i,k_m}|\leq C(h,\tau)r_m^{1-s}\left[ f_{i,k_m}\right]_{C^{1-s}_p(Q_1\cap\Omega_{k_m})},$$ 
	which is bounded by assumption.
	Hence
	$$\left|(\partial_t + L_{k_m})(v_{1,m}(x+h,t+\tau) - v_{1,m}(x,t))\right| \leq \frac{C(1+q_{r_m,k_m})}{\theta(r_m)},$$
	which goes to $0$ as $m\to \infty.$
	Note that the assumption on uniform bounds on $\left[ K_m\right] _{C^{1-s}(\S^{n-1})}$ and $\|K_m\|_{L^\infty(\S^{-1})}$ assure that up to a subsequence $K_m|_{\S^{n-1}}$ converge uniformly.
	Hence we can pass to the limit with \cite[Lemma 3.1]{RV18}, to get
	$$(\partial_t + L_{0})(w(x+h,t+\tau) - w(x,t)) =0,\quad\quad \text{whenever }x_n>0.$$
	Passing the growth control of $v_m$ to the limit gives $\left[ \frac{w}{(x_n)_+^s}\right]_{C^{s-2\varepsilon}(Q_R\cap{x_n>0})}\leq R^{1-s},$ and hence \cite[Theorem 2.1]{RV18} implies that $w=q(x_n)_+^s$, and hence $v = q$. But passing to the limit quantities in \eqref{eq:contradictionQuantities}, we get that $\left[ v\right]_{C^{s-2\varepsilon}_p(Q_1\cap\{x_n>0\})}>\frac{1}{2}$ but also $\int_{Q_1\cap\{x_n>0\}} v = 0$. These contradict each other, since $v$ is a constant function.
	
	To prove the statement in the case when $u_1\not\equiv0$ outside $Q_1$, we take a cut-off function $\chi\in C^\infty_c(B_1)$ which is $1$ in $B_{3/4}$ and apply the already proven case on $\frac{1}{C_1}u_1\chi$ and $u_2$, where $C_1$ is big enough so that $\left[ L(u_1\chi)\right] _{C^{1-s}_p(Q_1)}\leq C_1.$ We can derive the last estimate in the same way as in the proof of Lemma \ref{lem:interiorRegularityWithGrowth} -- equation \eqref{eq:cutoff} and the estimates below. 
\end{proof}

\begin{remark}
	We actually need not work with the decay of the seminorms of the quotient in order to establish expansions of order $1+s$. We nevertheless presented the result, since we believe they might be important for establishing expansions of higher orders and moreover we correct some imprecisions from \cite[Proposition 3.3]{RV18}.
\end{remark}

As a consequence we get $C^{1-\varepsilon}$ regularity of the quotient of two solutions up to the boundary.

\begin{corollary}\label{cor:1-epsRegOfQuotient}
	Let $s\in(\frac{1}{2},1)$. Let $\Omega\subset \R^{n+1}$ be $C^{1}_p$ in $Q_1$. Let $\gamma,\varepsilon>0$. Let $L$ be an operator of the form \eqref{eq:operatorForm}, with kernel $K\in C^{1-s}(\S^{n-1})$. Assume $u_i\in C_p^{\gamma}(Q_1)$, $i\in\{1,2\}$, solve
	$$\left\lbrace\begin{array}{rcll}
	(\partial_t + L)u_i &=& f_i& \text{in }\Omega\cap Q_1\\
	u_i&=&0&\text{in }\Omega^c\cap Q_1,
	\end{array}\right. $$
	with $\left[ f_i\right]_{C^{1-s}_p(\Omega\cap Q_1)}\leq 1$, $\left[ u_i\right] _{C^{\frac{1-s}{2s}}_t(B_R\times(-1,1))}\leq R^{2s-1}$, and $||u_2||_{L^\infty(B_R\times(-1,1))}\leq R^s$ and $||u_1||_{L^\infty(B_R\times(-1,1))}\leq R^{3s-1-\varepsilon}$. Assume also that $u_2\geq c_0 d^s$, for some $c_0>0$.
	
	Then  
	$$\left|\left| \frac{u_1}{u_2}\right|\right| _{C_p^{1-\varepsilon}(\Omega\cap Q_{1/2})}\leq C,$$
	where $C>0$ depends only on $n,s,\varepsilon,c_0,G_0,\|K\|_{C^{1-s}(\S^{n-1})}$ and ellipticity constants.
\end{corollary}
\begin{proof}
	Proposition \ref{prop:decay2functions} in combination with \cite[Proposition 3.4]{RV18} give that for every $(z,\tau)\in\partial\Omega\cap Q_{1/2}$ there is a constant $q_{(z,\tau)}$ so that
	$$|u_1(x,t) - q_{(z,\tau)}u_2(x,t)|\leq C(|x-z|^{1+s-\varepsilon}+|t-\tau|^{\frac{1+s-\varepsilon}{2s}}). $$
	Analogously as in the proof of Corollary \ref{cor:expansionWithD}, in combination with interior and boundary regularity estimates  (Proposition \ref{prop:holderBoundaryEstimate} and Lemma \ref{lem:interiorRegularityWithGrowth})
	we conclude for any $(x_0,t_0)$, such that $d_p(x_0,t_0) = |(x_0,t_0)-(z,\tau)|_p\leq c_\Omega r$, and that $Q_{2r}(x_0,t_0)\subset\Omega\cap Q_1$ we have 
	$$\left[ u_1 - q_{(z,\tau)}u_2\right]_{C^{1+s-\varepsilon}_p(Q_r(x_0,t_0))}\leq C.$$
	We can apply Lemma \ref{lem:divisionLemma} to get that 
	$$\left[ \frac{u_1}{u_2} \right]_{C^{1-\varepsilon}_p(Q_r(x_0,t_0))}=\left[ \frac{u_1}{u_2} - q_{(z,\tau)}\right]_{C^{1-\varepsilon}_p(Q_r(x_0,t_0))}\leq C,$$
	where Lemma \ref{lem:interiorRegularityWithGrowth} provides the needed interior estimates for $u_2$.
	The result now follows from \cite[Lemma B.2]{K21b}.
\end{proof}

\subsection{Boundary Harnack in $C^\beta_p$ domains}
	
	We now present a result that is required for establishing boundary Harnack inequalities of higher orders. It shows how well we can approximate solutions with $d^s$ near the boundary and in the interior.
	
	\begin{corollary}\label{cor:expansionWithD}
		Let $s\in(\frac{1}{2},1)$. Let $\Omega\subset \R^{n+1}$ be $C^{\beta}_p$ in $Q_1$, for some $\beta > 2s$. Let $\varepsilon>0$. Let $L$ be an operator of the form \eqref{eq:operatorForm}. Assume that $u\in C_p^{\gamma}(\R^n\times (-1,1))$, $\gamma>0$, solves
		$$\left\lbrace\begin{array}{rcll}
		(\partial_t +  L)u &=& f& \text{in }\Omega\cap Q_1\\
		u&=&0&\text{in }\Omega^c\cap Q_1,
		\end{array}\right. $$
		with  $||d^{1-s}f||_{L^\infty(\Omega\cap Q_1)}\leq 1$, $||u||_{L^\infty(B_R\times(-1,1))}\leq R^s$.
		
		Then for every $(z,t_0)\in \partial \Omega \cap Q_{1/2}$ there exists a constant $q_{(z,t_0)}$, so that 
		$$\left[ u(x,t)-q_{(z,t_0)}d^s(x,t)\right]_{C^{s-\varepsilon}_p(Q_r(z,t_0))}\leq C r^{2s-1+\varepsilon},\quad \quad r\in(0,1/2).$$
		The constant $C$ depends only on $\beta,n,s,\varepsilon,G_0$ and ellipticity constants.
		
		Moreover, if additionally $\beta>1+s$ and $\left[ f\right]_{C^{\alpha}_p(\Omega\cap Q_1)}\leq 1$,  $\left[ u\right]_{C^\frac{\alpha}{2s}_t(B_R\times(-1,1))} \leq R^{3s-1-\alpha}$,  for some $\alpha\in(0,s-\varepsilon]$, and  $K\in C^{2\beta+1}(\S^{n-1})$  then for every $(x_0,t_0)\in \Omega\cap Q_{1/2}$, such that $d_x(x_0) = |x_0-z| = c_\Omega r,$ we have 
		\begin{equation*}
		\left[ u - q_{(z,t_0)} d^s \right] _{C^{\alpha+2s}_p(Q_r(x_0,t_0))}\leq Cr^{-1+s-\alpha},
		\end{equation*}
		with $C$ additionally depending on $\|K\|_{C^{2\beta+1}(\S^{n-1})}.$
	\end{corollary}

	\begin{proof}
		From Corollary \ref{cor:basicQuotientEstimate}, and the fact that $(\partial_t+L)d^s$ is bounded by $d^{s-1}$ (see \cite[Lemma 3.2]{K21a}) we get
		$$\left[ \frac{u}{d^s}\right] _{C^{2s-1}_p(\Omega\cap Q_{3/4})}\leq C,\quad\quad|u(x,t)-q_{(z,t_0)}d^s(x,t)|\leq C \left(|x-z|+|t-t_0|^{\frac{1}{2s}}\right)^{3s-1}.$$
		Moreover Corollary \ref{cor:CsRegularity} gives that $u\in C^s_p(Q_{3/4}).$
		Since $u$ and $d^s$ grow at most as power $s$ at infinity, the second estimate implies that $||u-q_{(z,t_0)}d^s||_{L^\infty(Q_r(z,t_0))}\leq Cr^{3s-1}$ for all $r>0$. Using rescaled Proposition \ref{prop:holderBoundaryEstimate} gives
		\begin{align*}
			\left[ u(x,t)-q_{(z,t_0)}d^s(x,t)\right]_{C^{s-\varepsilon}_p(Q_r(z,t_0))}\leq& Cr^{-s+\varepsilon}\Big(  r^{2s}||d^s(f + (\partial_t + L)d^s)||_{L^\infty(Q_{2r}(z,t_0))} \\
			&\quad\quad + \sup_{R>1}R^{-2s+\varepsilon} ||u-q_{(z,t_0)}d^s||_{L^\infty(Q_{Rr}(z,t_0))}\Big)  \\
			\leq & C r^{2s-1+\varepsilon},
		\end{align*}
		whenever $r<\frac{1}{2}$. 
%		Combining it with the assumptions on $u$, we get
%		$$\left[ u(x,t)-q_{(z,t_0)}d^s(x,t)\right]_{C^{\frac{\alpha}{2s}}_t(Q_r(z,t_0))}\leq C r^{\varepsilon_0+s-\alpha},$$
%		for all $r>0$.
		
		To prove the moreover case, we use Lemma \ref{lem:interiorRegularityWithGrowth}, which gives
		\begin{align*}
			\left[ u - q_{(z,t_0)} d^s \right]_{C^{\alpha+2s}_p(Q_r(x_0,t_0))}\leq C&r^{-\alpha-2s}\Big( r^{\alpha+2s}\left[f - q_{(z,t_0)}(\partial_t+L)d^s \right]_{C^{\alpha}_p(Q_{2r}(x_0,t_0))}  \\
			& +\sup_{R>1}R^{-2s-\alpha+\varepsilon} ||u - q_{(z,t_0)} d^s ||_{L^\infty(Q_{2(R+2)r}(x_0,t_0))}  \\
			&+  \sup_{R>1}R^{-2s+\varepsilon}r^{\alpha}\left[ u - q_{(z,t_0)} d^s \right]_{C^{\alpha}_p(Q_{2r(R+2)})} \Big).
		\end{align*}
		Lemma \ref{lem:generalisedDistanceExsistence} and Proposition \ref{prop:Lds} assure that 
		$$\left[(\partial_t+L)d^s \right]_{C^{s-\varepsilon}_p(Q_{2r}(x_0,t_0))}\leq Cr^{-1+\varepsilon},$$
		while above we established that 
		$$||u - q_{(z,t_0)} d^s ||_{L^\infty(Q_{2(R+2)r}(x_0,t_0))}\leq C ((R+2)r)^{3s-1}.$$
		Finally, combining the already proven estimate  with the assumptions on $u$, we get
		$$\left[ u(x,t)-q_{(z,t_0)}d^s(x,t)\right]_{C^{\frac{\alpha}{2s}}_t(Q_r(z,t_0))}\leq C r^{3s-1-\alpha},$$
		for all $r>0$, which implies
		$$\left[ u - q_{(z,t_0)} d^s \right]_{C^{\frac{\alpha}{2s}}_t(Q_{2r(R+2)})}\leq C (r(R+2))^{3s-1-\alpha}.$$
		We conclude that 
		$$\left[ u - q_{(z,t_0)} d^s \right]_{C^{\alpha+2s}_p(Q_r(x_0,t_0))}\leq Cr^{-1+s-\alpha},$$
		which proves the claim.
	\end{proof}

	In order to get the expansions for two solutions of orders exceeding $1+s$, we prove the decay of seminorms of the expansion. Since the decay rate is transformed in the growth of the blow up sequence, it can not be taken larger than $2s$. In combination with the $C^s$ nature of solutions, with this strategy we can only achieve expansions of order $3s$. The precise version of the claim is stated below.

\begin{proposition}\label{prop:expansion2functions3s}
	Let $s\in(\frac{1}{2},1)$. Let $\Omega\subset \R^{n+1}$ be $C^{\beta}_p$ in $Q_1$, for some $\beta>1+s $. 
	Let $\varepsilon>0$ and $\gamma>0$. Let $L$ be an operator of the form \eqref{eq:operatorForm}, with kernel $K\in C^{2\beta+1}(\S^{n-1})$. Assume $u_i\in C_p^{\gamma}(Q_1)$, $i\in\{1,2\}$, solve
	$$\left\lbrace\begin{array}{rcll}
	(\partial_t + L)u_i &=& f_i& \text{in }\Omega\cap Q_1\\
	u_i&=&0&\text{in }\Omega^c\cap Q_1,
	\end{array}\right. $$
	with $\left[ f_i\right]_{C^{\alpha}_p(\Omega\cap Q_1)}\leq 1$, $||u_i||_{L^\infty(B_R\times(-1,1))}\leq R^{2s+\alpha-\varepsilon}$ and $\left[ u_i\right]_{C^{\frac{\alpha}{2s}}_t(B_R\times(-1,1))}\leq R^{2s-\varepsilon}$, for some $\alpha\in(1-s,s-\varepsilon ] $.
	Assume also that $u_2\geq c_0 d^s$, for some $c_0>0$.
	
	Then for every $(z,\tau)\in\partial\Omega\cap Q_{1/2}$ there are $q_{(z,\tau)}\in \R$ and $Q_{(z,\tau)}$ a polynomial of degree $1$ in $\R^n$ with $Q_{(z,\tau)}(z,\tau)=0$, so that
	$$\left[ u_1 - q_{(z,\tau)} u_2 - Q_{(z,\tau)}  d^s \right]_{C^{\alpha}_p(Q_r(z,\tau))}\leq Cr^{2s-\varepsilon}.$$
	The constant $C>0$ depends only on $n,s,\varepsilon,c_0,G_0$ and ellipticity constants.
\end{proposition}
\begin{proof}
	Thanks to the assumption on the domain, we can assume that $(z,\tau)=(0,0)$. We can also assume $u_i = 0$ outside of $Q_1$, $u\in C^{s}_p(\R^n\times(-1,1))$,\footnote{We achieve it with multiplying it with a smooth cut off, which does not change the right-hand side too much.} thanks to the growth control, global regularity in time and regularity of the kernel (see Lemma \ref{lem:interiorRegularityWithGrowth}). We argue by contradiction. Assume that there exist $\Omega_k,L_k,u{i,k},f_{i,k}$, $i=1,2$, as in the statement, so that 
	$$\sup_k\sup_{r>0} r^{-2s+\varepsilon}\left[ u_{1,k} - q_k u_{2,k} - Q_k  d_k^s \right]_{C^{\alpha}_p(Q_r)} = \infty,$$
	for any $q_k\in \R$ and $Q_k$ $1-$homogeneous polynomial. We define $q_{r,k},Q_{r,k}$ as the coefficients of $L^2(Q_r)$ projection of $u_{1,k}$ to $\R u_{2,k} + \textbf{P}d_k^s$, where $\textbf{P}$ stands for all $1$-homogeneous polynomials. Therefore we have
	$$\int_{Q_r}(u_{1,k}-q_{r,k}u_{2,k} - Q_{r,k}d_k^s)(qu_{2,k}+ Qd_k^s) = 0,\quad\quad \text{for all }q\in\R,\text{ } Q\in\textbf{P}.$$
	Furthermore we define a monotone quantity
	$$\theta(r) = \sup_k\sup_{\rho>r} \rho^{-2s+\varepsilon}\left[ u_{1,k} - q_{\rho,k} u_{2,k} - Q_{\rho,k}  d_k^s \right]_{C^{\alpha}_p(Q_\rho)}.$$
	From Lemma \ref{lem:fakeLemma4.5}  
	we deduce that $\theta(r)\to\infty$ as $r\downarrow0$.
	Hence we can get a sequences $r_m\downarrow0$ and $k_m$, for $m\in\N$, so that
	$$\frac{m}{4}\leq \frac{\theta(r_m)}{2} \leq r_m^{-2s+\varepsilon}\left[ u_{1,k_m} - q_{r_m,k_m} u_{2,k_m} - Q_{r_m,k_m}  d_{k_m}^s \right]_{C^{\alpha}_p(Q_{r_m})}\leq \theta(r_m).$$
	We define the blow-up sequence
	\begin{equation}\label{eq:contradictionProperties}
		v_m(x,t) = \frac{1}{\theta(r_m)r_m^{\alpha+2s-\varepsilon}}\left(u_{1,k_m} - q_{r_m,k_m}u_{2,k_m} - Q_{r_m,k_m}d_{k_m}^s\right)(r_mx,r_m^{2s}t).
	\end{equation}
	By definition of $q_{r,k}$ and choice of $r_m,k_m$ we have
	$$\int_{Q_1} v_m (qu_{2,m} - Qd_m^s) = 0,\quad\quad \frac{1}{2}\leq\left[ v_m\right]_{C^{\alpha}_p(Q_1)}\leq 1,$$
	for any $q\in\R$ and $Q\in \textbf{P}.$ 
	
	Next we want to obtain growth control near infinity for the blow-up sequence. In this direction we estimate the decay of coefficients of $Q_{r,k}$ and $q_{r,k}$ at zero. We denote $Q_{r,k}(x)=Q_{r,k}\cdot x,$ and proceed in the following way. First, by  rescaled \cite[Lemma A.10]{AR20}
	\begin{align*}
		|Q_{r,k}-Q_{2r,k}|\leq &  Cr^{-1} ||(Q_{r,k}-Q_{2r,k})||_{L^\infty(Q_r\cap\{d_k>r/2\})}\\ 
		\leq & Cr^{-1-s} ||Q_{r,k}d_k^s-Q_{2r,k}d_k^s||_{L^\infty(Q_r\cap\{d_k>r/2\})}\\
		\leq &Cr^{-1-s+\alpha} \left[ Q_{r,k}d_k^s-Q_{2r,k}d_k^s\right] _{C_p^\alpha(Q_r)}\\
		\leq & Cr^{-1-s+\alpha} \left( \left[ u_{1,k} - q_{r,k}u_{2,k} - Q_{r,k} d_k^s\right] _{C^{\alpha}_p(Q_r)} \right.\\
		&\left.\quad + \left[ u_{1,k} - q_{2r,k}u_{2,k} - Q_{2r,k} d_k^s\right] _{C^{\alpha}_p(Q_{2r})} \right.\\
		&\left.\quad + \left[ q_{r,k}u_{2,k} - q_{2r,k}u_{2,k} \right] _{C^{\alpha}_p(Q_{r})} \right)\\
		\leq & Cr^{-1-s+\alpha} \left( \theta(r)r^{2s-\varepsilon} + \theta(2r)(2r)^{2s-\varepsilon} +|q_{r,k}-q_{2r,k}|r^{s-\alpha} \right)\\
		\leq & C\left( \theta(r) r^{s+\alpha-1-\varepsilon} + |q_{r,k}-q_{2r,k}|r^{-1} \right),
	\end{align*}
	where we used that  $\left[ u_{2,k}\right]_{C^s_p}(Q_1)\leq C, $ see Corollary \ref{cor:CsRegularity}. Moreover from the definition of $\theta$ and the triangle inequality we get
	$$ \left[(q_{r,k}-q_{2r,k}) u_{2,k} +  (Q_{r,k}-Q_{2r,k})\cdot x d_k^s\right]_{C^{\alpha}_p(Q_r)}\leq C\theta(r)r^{2s-\varepsilon},$$
	which furthermore implies
	\begin{align*}
		||(q_{r,k}-q_{2r,k}) u_{2,k} +  (Q_{r,k}-Q_{2r,k})&\cdot x d_k^s||_{L^{\infty}(Q_r\cap\{d_k<r/2\})}\leq \\
		&\leq ||(q_{r,k}-q_{2r,k}) u_{2,k} +  (Q_{r,k}-Q_{2r,k})\cdot x d_k^s||_{L^{\infty}(Q_r)}\\
		&\leq r^{\alpha} \left[(q_{r,k}-q_{2r,k}) u_{2,k} +  (Q_{r,k}-Q_{2r,k})\cdot x d_k^s\right]_{C^{\alpha}_p(Q_r)}\\
		&\leq C\theta(r)r^{2s+\alpha-\varepsilon},
	\end{align*}
	and hence 
	$$\left|\left|(q_{r,k}-q_{2r,k}) \frac{u_{2,k}}{d_k^s} +  (Q_{r,k}-Q_{2r,k})\cdot x \right|\right|_{L^{\infty}(Q_r\cap\{d_k<r/2\})}\leq C\theta(r)r^{s+\alpha-\varepsilon}.$$
	It follows from \cite[Lemma A.11]{AR20} that
	$$|q_{r,k}-q_{2r,k}|\leq C\theta(r)r^{s+\alpha-\varepsilon},$$
	which implies  
	$$|Q_{r,k}-Q_{2r,k}|\leq C\theta(r)r^{s+\alpha-1-\varepsilon}.$$
	In the same way as in \cite[Proposition 4.4]{AR20} we deduce that 
	\begin{align*}
		|q_{r,k}-q_{Rr,k}|\leq C\theta(r)(Rr)^{s+\alpha-\varepsilon}\\
		|Q_{r,k}-Q_{Rr,k}|\leq C\theta(r)(Rr)^{s+\alpha-1-\varepsilon}\\
		\frac{|q_{r,k}|+|Q_{r,k}|}{\theta(r)}\downarrow 0,\quad\text{ uniformly in }k.
	\end{align*}
	This implies the growth control of the blow-up sequence
	\begin{align}\label{eq:growthControlOfBlowUpSequence}
	\begin{split}
		\left[ v_m\right]_{C^{\alpha}_p(Q_R)}\leq& \frac{r_m^{\alpha}}{\theta(r_m)r_m^{\alpha+2s-\varepsilon}} \left( \theta(Rr_m)(Rr_m)^{2s-\varepsilon} + |q_{r_m,k_m}-q_{Rr_m,k_m}| \left[ u_{2,k_m}\right]_{C^{\alpha}(Q_{Rr_m})} \right.\\
		&\quad\quad\quad\quad+\left. |Q_{r_m,k_m}-Q_{Rr_m,k_m}| \left[ xd_{k_m}^s\right]_{C^{\alpha}(Q_{Rr_m})} \right)\\
		\leq& \frac{Cr_m^{\alpha}}{\theta(r_m)r_m^{\alpha+2s-\varepsilon}} \left( \theta(r_m)(Rr_m)^{2s-\varepsilon} + \theta(r_m)(Rr_m)^{\alpha +s-\varepsilon} (Rr_m)^{s-\alpha} \right.\\
		&\quad\quad\quad\quad+\left. \theta(r_m)(Rr_m)^{s+\alpha-1-\varepsilon} (Rr_m)^{1+s-\alpha} \right)\\
		\leq & C R^{2s-\varepsilon}.
	\end{split}
	\end{align}
	Note that the estimate is valid for all $R>1$,\footnote{When $R>r_m^{-1}$, then $Q_R$ has to be intersected with $\R^n\cap (-1,1)$.}
	and that we used $\left[ u_{2,k_m}\right]_{C^s_p(\R^n\times(-1,1))}\leq 1$.
	
	Denoting $\Omega_m = \{(x,t); \text{ } (r_mx,r_m^{2s}t)\in\Omega_{k_m} \}$, we have that $v_m$ solves
	\begin{align*}
		(\partial_t + L_{k_m})v_m(x,t) = \frac{1}{\theta(r_m)r_m^{\alpha-\varepsilon}}&\big(f_{1,k_m}- q_{r_m,k_m}f_{2,k_m} 
		- sQ_{r_m,k_m}d_{k_m}^{s-1}\partial_td_{k_m}\\
		&  - L_{k_m}(Q_{r_m,k_m}d_{k_m}^s) \big)(r_mx,r_m^{2s}t)
	\end{align*}
	in $\Omega_m\cap Q_{r_m^{-1}}$, as well as 
	$$	v_m=0, \quad\quad \text{ in }\Omega_m^c\cap Q_{r_m^{-1}}.$$
	By assumption we can bound
	%$$\frac{1}{\theta(r_m)r_m^{\alpha-\varepsilon}} \left[f_{1,k_m}(r_m\cdot, r_m^{2s}\cdot) \right]_{C^{\alpha-\varepsilon}_p(\Omega_m\cap Q_{r_{k_m}^{-1}})}\leq  C\frac{1}{\theta(r_{k_m})},$$
	\begin{align*}
		\frac{1}{\theta(r_m)r_m^{\alpha-\varepsilon}} \left[f_{1,k_m}(r_m\cdot, r_m^{2s}\cdot) \right]_{C^{\beta-1-s}_p(\Omega_m\cap Q_{2})}\leq& \text{ } C\frac{r_m^{\beta-1-s}}{\theta(r_{k_m})r_m^{s-\varepsilon}}
		\left[f_{1,k_m} \right]_{C^{\beta-1-s}_p(\Omega_{k_m}\cap Q_{2r_m})}\\
		\leq &\text{ } C\frac{1}{\theta(r_{k_m})}
		\left[f_{1,k_m} \right]_{C^{\alpha-\varepsilon}_p(\Omega_{k_m}\cap Q_{2r_m})}\\
		\leq &\text{ } C\frac{1}{\theta(r_{k_m})}
	\end{align*}
	which goes to zero and in particular it is bounded for all $m$. Bounding the term with $f_{2,k_m}$, we obtain $C\frac{|q_{r_m,k_m}|}{\theta(r_m)}$ which also converges to zero.
	We proceed with estimating
	\begin{align*}
		\left|\frac{1}{\theta(r_m)r_m^{\alpha-\varepsilon}} Q_{r_m,k_m}d_{k_m}^{s-1}\partial_td_{k_m} (r_mx,r_m^{2s}t)\right| &\leq \frac{C|Q_{r_m,k_m}|}{\theta(r_m)r_m^{\alpha-\varepsilon}}|r_mx|r_m^{s-1} d_m^{s-1}(x,t)\\
		&\leq C\frac{|Q_{r_m,k_m}|}{\theta(r_m)}d_m^{s-1}(x,t),
	\end{align*}
	where $d_m$ stands for the distance function in $\Omega_m$. Finally by Proposition \ref{prop:Lds} we also get that
	\begin{align*}
		\frac{1}{\theta(r_m)r_m^{\alpha-\varepsilon}}&\left[ L_{k_m}(Q_{r_m,k_m}d_{k_m}^s)(r_m\cdot,r_m^{2s}\cdot)\right]_{C^{\beta-1-s}_p(\Omega_m\cap Q_{2})}\leq\\
		&\leq \frac{r_m^{\beta-1-s}}{\theta(r_m)r_m^{\alpha-\varepsilon}}\left[ L_{k_m}(Q_{r_m,k_m}d_{k_m}^s)\right]_{C^{\beta-1-s}_p(\Omega_{k_m}\cap Q_{2r_m})} \\
		&\leq C\frac{|Q_{r_m,k_m}|r_m^{s-\varepsilon}}{\theta(r_m)r_m^{\alpha-\varepsilon}}\\
		&\leq C\frac{|Q_{r_m,k_m}|}{\theta(r_m)},
	\end{align*}
	which also converges to $0$ as $m\to\infty.$ 
	Hence Corollary \ref{cor:advancedBoundaryReg} gives that 
	$$\left[ v_m\right]_{C_p^{\alpha+\varepsilon/2}(Q_1)}\leq C,$$
	independently of $m$. Note that boundedness of $||v_m||_{L^\infty(Q_2)}$ follows from $v_m(0,0)=0$ and the uniform control on the seminorm $C_p^{\alpha}(Q_2)$, see \eqref{eq:growthControlOfBlowUpSequence}. Moreover the growth control \eqref{eq:growthControlOfBlowUpSequence} and the newly obtained estimate together with Arzela-Ascoli theorem give that up to passing to a subsequence $v_m$ converges locally uniformly in $\R^{n+1}$ to some function $v$, and the convergence is $C^{\alpha}_p$ in $Q_1$. Choosing $h\in\R^n$ with $h_n\geq 0$ and $\tau\in\R$, and denoting $v_h(x,t) = v(x+h,t+\tau)$ we conclude from \cite[Proposition 3.1]{RV18}, that
	$$\left\lbrace\begin{array}{rcll}
		(\partial_t + L )(v_h - v) &=&0&\text{ in }\{x_n>0\}\\
		v&=&0&\text{ in }\{x_n\leq 0\}\\
		\left[ v\right]_{C^{\alpha}_p(Q_R)}&\leq & CR^{2s-\varepsilon}&\text{ for all }R>1.
	\end{array}\right. $$
	Hence the Liouville theorem (Proposition \ref{prop:Liouville}) yields that $v(x,t) = (x_n)_+^s(q+Q\cdot x)$, for some $Q\in\textbf{P}.$ 
	But passing the quantities in \eqref{eq:contradictionProperties} to the limit, we get that $q=0$ and $Q = 0$, which contradicts $\left[ v\right]_{C^{\alpha}(Q_1)}\geq \frac{1}{2}.$
\end{proof}

The Liouville type result used above reads as follows.

\begin{proposition}\label{prop:Liouville}
	Let $s\in(0,1)$, $\alpha\in(0,s),$ and $\beta\in(0,2s)$. Assume that $w$ satisfies
	$$\left\lbrace\begin{array}{rcll}
		(\partial_t+ L )(w(\cdot+h,\cdot+\tau)-w) &=&0&\text{in }\{x_n>0\}\times(-\infty,0)\\
		w&=&0&\text{in }\{x_n\leq0\}\times(-\infty,0),
	\end{array}\right.$$
	where $L$ is an operator of the form \eqref{eq:operatorForm}, $h\in\R^n$ with $h_n\geq0$ and $\tau<0$. Assume that $w$ satisfies the growth condition
	$$\left[ w\right]_{C^\alpha_p(Q_R\cap \{t<0\})}\leq R^{\beta},$$
	for $R>1$. Then
	$$w(x,t) = (x_n)_+^s(p\cdot x + q),$$
	for some $p\in\R^n$ and $q\in\R$. 
\end{proposition}
\begin{proof}
	Let first $h_n=0$. Denote $v = w(\cdot+h,\cdot+\tau)-w.$ We have
	$$\left\lbrace\begin{array}{rcll}
	(\partial_t+ L )v &=&0&\text{in }\{x_n>0\}\times(-\infty,0)\\
	v&=&0&\text{in }\{x_n\leq0\}\times(-\infty,0),
	\end{array}\right.$$ 
	and $||v||_{L^\infty(Q_R\cap \{t<0\})}\leq CR^\beta.$ Applying \cite[Theorem 4.11]{FR17}, we conclude that $v(x)=K(x_n)_+^s$ for some constant $K$. The rest follows in the same way as in \cite[Theorem~2.1]{RV18}.
\end{proof}

We are now well equipped to prove $C^{2s-\varepsilon}$ regularity of the quotient two solutions. 

\begin{corollary}\label{cor:expansion2functions3s}
	Let $s\in(\frac{1}{2},1)$. Let $\Omega\subset \R^{n+1}$ be $C^{\beta}_p$ in $Q_1$, for some $\beta>1+s$.
	Let $\varepsilon>0$. Let $L$ be an operator of the form \eqref{eq:operatorForm}, with kernel $K\in C^{2\beta+1}(\S^{n-1})$. Assume $u_i\in C_p^{\gamma}(Q_1)$, $i\in\{1,2\}$, $\gamma>0$, solve
	$$\left\lbrace\begin{array}{rcll}
	(\partial_t + L)u_i &=& f_i& \text{in }\Omega\cap Q_1\\
	u_i&=&0&\text{in }\Omega^c\cap Q_1,
	\end{array}\right. $$
	with $\left[ f_i\right]_{C^{\alpha}_p(\Omega\cap Q_1)}\leq 1$, $||u_i||_{L^\infty(B_R\times(-1,1))}\leq R^{2s+\alpha-\varepsilon}$, $\left[ u_1\right]_{C^{\frac{\alpha}{2s}}_t(B_R\times(-1,1))}\leq R^{2s-\varepsilon}$ and $\left[ u_2\right]_{C^{\frac{\alpha}{2s}}_t(B_R\times(-1,1))}\leq R^{3s-1-\alpha}$, for some $\alpha\in(1-s,s-\varepsilon ] $.
	Assume also that $u_2\geq c_0 d^s$, for some $c_0>0$.
	
	Then 
	$$ \left|\left| \frac{u_1}{u_2} \right|\right|_{C^{s+\alpha-\varepsilon}_p(\Omega\cap Q_{1/2})}  \leq C.$$
	The constant $C>0$ depends only on $n,s,\alpha,\varepsilon,c_0,G_0,\|K\|_{C^{2\beta+1}(\S^{n-1})}$ and ellipticity constants.
\end{corollary}

\begin{proof}
	Let $(x_0,t_0)\in \Omega\cap Q_{1/2}$ be such that $d_x(x_0,t_0)\leq C_0 r_0$ and that $Q_{2r_0}(x_0,t_0)\subset \Omega.$ Let $(z,t_0)\in\partial\Omega$ be the closest point to $(x_0,t_0)$ in the boundary at time $t_0$. We want to show that  $ \left[ \frac{u_1}{u_2} \right]_{C^{\alpha+s}_p(Q_{r_0}(x_0,t_0)}  \leq C,$ with $C$ not depending on $x_0,t_0,r_0$.
	From Proposition \ref{prop:expansion2functions3s} we get that for  $(z,t_0)$ we have an expansion of the form
	$$\left[ u_1 - q_{(z,t_0)} u_2 - Q_{(z,t_0)}  d^s \right]_{C^{\alpha}_p(Q_r(z,t_0))}\leq Cr^{2s-\varepsilon}.$$ 
	In combination with $( u_1 - q_{(z,t_0)} u_2 - Q_{(z,t_0)}  d^s)(z,t_0)=0$, we conclude 
	$$||u_1 - q_{(z,t_0)} u_2 - Q_{(z,t_0)}  d^s||_{L^\infty(Q_r(z,t_0))}\leq C r^{\alpha+2s-\varepsilon}.$$
	Moreover since by assumption on the growth of $|| u_i|| _{L^\infty(B_R\times(-1,1))}$ and  $\left[ u_i\right] _{C^\frac{\alpha}{2s}_t(B_R\times(-1,1))}$ we have 
	$$\left[ u_1 - q_{(z,t_0)} u_2 - Q_{(z,t_0)}  d^s \right]_{C^{\frac{\alpha}{2s}}_t(Q_r(z,t_0))}\leq Cr^{2s-\varepsilon},$$
	and 
	$$||u_1 - q_{(z,t_0)} u_2 - Q_{(z,t_0)}  d^s||_{L^\infty(Q_r(z,t_0))}\leq C r^{\alpha+2s-\varepsilon}$$
	for all $r>0$.
	
	We now apply Lemma \ref{lem:interiorRegularityWithGrowth} on the function $v_{r_0}(x,t) = (u_1 - q_{(z,t_0)} u_2 - Q_{(z,t_0)}  d^s)(x_0+2r_0x,t_0+(2r_0)^{2s}t),$ to deduce that 
	\begin{align*}
		\left[ u_1 - q_z u_2 - Q_z  \right.&\left.d^s \right]_{C^{\alpha+2s-\varepsilon}_p(Q_{r_0}(x_0,t_0))}\leq Cr_0^\varepsilon \left[ u_1 - q_z u_2 - Q_z d^s \right]_{C^{\alpha+2s}_p(Q_{r_0}(x_0,t_0))}\\
		& \quad\quad\quad\quad\quad\quad\quad\quad\quad\leq Cr_0^{-\alpha-2s+\varepsilon}\left( r_0^{\alpha+2s}\left[ f_1 + q_{(z,t_0)}f_2\right]_{C^{\alpha}_p(Q_{2r_0}(x_0,t_0))} \right.\\
		&\left.+ r_0^{\alpha+2s}\left[(\partial_t + L)(Q_{(z,t_0)}d^s)\right]_{C^{\alpha}_p(Q_{2r_0}(x_0,t_0))} \right.\\
		&\left.+ \sup_{R>1} R^{-2s-\alpha+\varepsilon}||u_1 - q_{(z,t_0)} u_2 - Q_{(z,t_0)}  d^s||_{L^\infty(Q_{2(R+2)r_0}(x_0,t_0))}  \right.\\
		&\left. +\sup_{R>1} R^{-2s+\varepsilon} r_0^{\alpha}\left[ u_1 - q_{(z,t_0)} u_2 - Q_{(z,t_0)}  d^s \right]_{C^{\frac{\alpha}{2s}}_t(Q_{2(R+2)r_0}(z,t_0))}  
		\right),
	\end{align*}
	where we denoted $q_z,Q_z = q_{(z,t_0)},Q_{(z,t_0)}$ in the first line for transparency.
	The seminorms of $f_i$ are bounded by assumption, $q_{(z,t_0)}$ and $Q_{(z,t_0)}$ are bounded  thanks to \cite[Lemma B.5]{K21a}. The term with $\partial_t(Q_{(z,t_0)}d^s)$ we treat in the following way
	\begin{align*}
		\left[ Q_{(z,t_0)}\partial_t d d^{s-1} \right]_{C^{s-\varepsilon}_p(Q_{2r_0}(z,t_0))}&\leq \left[ Q_{(z,t_0)}\right] ||\partial_t d d^{s-1}|| + \left[ \partial_t d\right]  ||Q_{(z,t_0)} d^{s-1}||+ \left[ d^{s-1}\right] ||Q\partial_t d||\\
		&\leq C(r_0^{1-s+\varepsilon} r_0^{s-1} + r_0^{(\beta-2s+\varepsilon) - s}r_0^{s} + r_0^{-1+\varepsilon}r_0    ) \\
		&\leq C,
	\end{align*}
	thanks to Lemma \ref{lem:generalisedDistanceExsistence}. The estimate  
	$$\left[L(Q_{(z,t_0)}d^s)\right]_{C^{s-\varepsilon}_p(Q_{2r_0}(x_0,t_0))}\leq C$$
	is provided by Proposition \ref{prop:Lds}, while 
	$$||u_1 - q_{(z,t_0)} u_2 - Q_{(z,t_0)}  d^s||_{L^\infty(Q_{2Rr_0}(x_0,t_0))} \leq C ((R+2)r_0)^{\alpha+2s-\varepsilon}$$
	and 
	$$\sup_{R>1} R^{-2s+\varepsilon} r_0^{\alpha}\left[ u_1 - q_{(z,t_0)} u_2 - Q_{(z,t_0)}  d^s \right]_{C^{\frac{\alpha}{2s}}_t(Q_{2(R+2)r_0}(z,t_0))}\leq Cr_0^{\alpha+2s-\varepsilon} $$
	are provided above. 
	Hence we have
	\begin{equation}\label{eq:interiorWithDs}
		\left[ u_1 - q_{(z,t_0)} u_2 - Q_{(z,t_0)}  d^s \right]_{C^{\alpha+2s-\varepsilon}_p(Q_{r_0}(x_0,t_0))}\leq C.
	\end{equation}
	Moreover by Corollary \ref{cor:expansionWithD} we have 
	%Here we are being suboptimal... this inequality is just needed for \alpha+2s-$\varepsilon$, which can relax the assumption on the growth of time continuity..... on u_2. But WE.
	$$\left[ u_2 - q_{2,(z,t_0)}  d^s \right]_{C^{\alpha+2s}_p(Q_{r_0}(x_0,t_0))}\leq Cr^{-1+s-\alpha},\quad \quad \left[ u_2 - q_{2,(z,t_0)}  d^s \right]_{C^{3s-1}_p(Q_{r_0}(x_0,t_0))}\leq C.$$
	Note that the first inequality also implies that 
	$$||\partial_t(u_2 - q_{2,(z,t_0)}  d^s)||_{L^\infty(Q_{r_0}(x_0,t_0))}\leq Cr^{s-1},$$
	see \cite[Lemma A.5]{K21a}. Hence for any polynomial $Q$ of degree $1$ with $Q(z)=0$ we have
	\begin{align*}
		\left[ Q(u_2-d^s) \right]_{C^{\alpha+2s}_p(Q_{r_0}(x_0,t_0))}\leq  & ||Q||_{L^\infty} \left[ u_2 - q_{2,(z,t_0)}  d^s \right]_{C^{\alpha+2s}_p(Q_{r_0}(x_0,t_0))}\\
		&+|\nabla Q|  \left[ u_2 - q_{2,(z,t_0)}  d^s \right]_{C^{\alpha+2s-1}_p(Q_{r_0}(x_0,t_0))}\\
		& + \left[ Q\right]_{C^{\alpha}_p(Q_{r_0}(x_0,t_0))}||\partial_t(u_2 - q_{2,(z,t_0)}  d^s)||_{L^\infty(Q_{r_0}(x_0,t_0))}\\
		\leq & Cr^{1}r^{-1+s-\alpha} + C r^{s-\alpha} + Cr^{1-\alpha} r^{s-1}\\
		\leq & C.
	\end{align*}
	Combining it with \eqref{eq:interiorWithDs}, we get
	$$\left[ u_1 - \tilde{Q}_{(z,t_0)}  u_2 \right]_{C^{\alpha+2s-\varepsilon}_p(Q_{r_0}(x_0,t_0))}\leq C,$$
	for $\tilde{Q}_{(z,t_0)} (x) = Q_{(z,t_0)}(0) + q_{2,(z,t_0)}^{-1} \nabla Q_{(z,t_0)} \cdot(x-z) $. Using Lemma \ref{lem:divisionLemma} we conclude
	$$\left[ \frac{u_1}{u_2} - \tilde{Q}_{(z,t_0)}  \right]_{C^{\alpha+s-\varepsilon}_p(Q_{r_0}(x_0,t_0))}\leq C,$$
	which proves the claim in view of \cite[Lemma B.2]{K21b}.
\end{proof}

To conclude this section we prove Theorem \ref{thm:bdryHarnack}.

\begin{proof}[Proof of Theorem \ref{thm:bdryHarnack}]
	It is a direct consequence of Corollary \ref{cor:1-epsRegOfQuotient} and  Corollary \ref{cor:expansion2functions3s}.
\end{proof}

\subsection{Optimal H\"older estimates}

Using the expansion result (Proposition~\ref{boundaryHarnackExpansion}) on a solution $u$ and  $d^s$, yields that solutions grow at most as $d^s$ near the boundary. In combination with interior regularity results, we can prove that the solutions must be $C^s$ up to the boundary (in space and time, which is better than $C^s_p$).

\begin{corollary}\label{cor:CsRegularity}
	Let $s\in(\frac{1}{2},1)$. Let $\Omega\subset \R^{n+1}$ be $C^{\beta}_p$ in $Q_1$, for some $\beta > 2s$. Let $\varepsilon>0$. Let $L$ be an operator of the form \eqref{eq:operatorForm}. Assume that $u\in C_p^{\gamma}(\R^n\times (-1,1))$, $\gamma>0$, solves
	$$\left\lbrace\begin{array}{rcll}
	(\partial_t +  L)u &=& f& \text{in }\Omega\cap Q_1\\
	u&=&0&\text{in }\Omega^c\cap Q_1,
	\end{array}\right. $$	
	Then 
	$$\left|\left| u\right|\right|_{C^s(Q_{1/2})}\leq C\left(||fd^{1-s}||_{L^\infty(Q_1\cap\Omega)} + \sup_{R>1} R^{-2s+\varepsilon}||u||_{L^\infty(B_R\times (-1,1))}\right),$$
	and
	$$\left|\left| \frac{u}{d^s}\right|\right|_{C^{2s-1}_p(Q_{1/2}\cap\overline{\Omega})}\leq C\left(||fd^{1-s}||_{L^\infty(Q_1\cap\Omega)} + \sup_{R>1} R^{-2s+\varepsilon}||u||_{L^\infty(B_R\times (-1,1))}\right),$$
	where $C$ depends only on $n,s,G_0$ and ellipticity constants.
\end{corollary}
\begin{proof}
	Dividing $u$ with $\left(||fd^{1-s}||_{L^\infty(Q_1\cap\Omega)} + \sup_{R>1} R^{-2s+\varepsilon}||u||_{L^\infty(Q_R\times (-1,1))}\right)$, it suffices to prove that $\left[ u\right]_{C^s(Q_{1/2})}\leq C$ and $\left[ \frac{u}{d^s}\right]_{C^{2s-1}_p(Q_{1/2})}\leq 1$, when it holds $||fd^{1-s}||_{L^\infty(Q_1\cap\Omega)} + \sup_{R>1} R^{-2s+\varepsilon}||u||_{L^\infty(Q_R\times (-1,1))}\leq1$.
	
	Note that $(\partial_t+L)d^s$ is bounded by $d^{s-1}$ (see \cite[Lemma 3.2]{K21a}), so from Corollary \ref{cor:basicQuotientEstimate} we get
	$$\left[ \frac{u}{d^s}\right] _{C^{2s-1}_p(\Omega\cap Q_{1/2})}\leq C.$$
	Moreover by Proposition \ref{boundaryHarnackExpansion}
	\begin{equation}\label{eq:oscDecayUDS}
		\left[ u-q_{(z,t_0)}d^s\right]_{C^{3s-1}_p(Q_r(x_0,t_0))}\leq C ,
	\end{equation}
	whenever $Q_{2r}(x_0,t_0)\subset\Omega$ and $d(x_0,t_0) \leq 2r$. 
	From this we deduce that 
	$$\left[ u\right]_{C_p^{s}(Q_r(x_0,t_0))}\leq C\left[ d^s\right]_{C_p^{s}(Q_r(x_0,t_0))}\leq C ,$$
	in view of \cite[Lemma B.5]{K21b}. 
	From \cite[Lemma B.2]{K21b} we get that 
	$$\left[ u\right]_{C_p^{s}(\Omega\cap Q_{1/2})}\leq C,$$
	and hence in particular
	$$\left[ u\right]_{C_x^{s}(\Omega\cap Q_{1/2})}\leq C.$$
	
	To get the time regularity of solutions we need to work a bit more. Choose $(x,t),(x,t')\in\Omega\cap Q_{1/2}$. First if $|t-t'|^{\frac{1}{2s}}\leq \frac{1}{2} \max\{d(x,t),d(x,t')\}$, then we can find $(x_0,t_0)\in\Omega$, so that $(x,t),(x,t')\in Q_{2r}(x_0,t_0)\subset\Omega$ and $d(x_0,t_0) \leq 2r$. Then from \eqref{eq:oscDecayUDS} we deduce that
	$$|u(x,t)-u(x,t')|\leq \left(\left[ u-q_{(z,t_0)}d^s\right]_{C^{s}_t(Q_r(x_0,t_0))}+\left[ d^s\right]_{C^{s}_t(Q_r(x_0,t_0))}\right)|t-t'|^{s}\leq C|t-t'|^s,$$
	since $\frac{3s-1}{2s}>s$ for $s\in(\frac{1}{2},1)$.
	Let now $|t-t'|^{\frac{1}{2s}}>\frac{1}{2} \max\{d(x,t),d(x,t')\}$. First notice that the regularity of the quotient implies also that 
	$$\left[ u-q_{(z,t_0)}d^s\right]_{C^{2s-1}_p(Q_r(z,t_0))}\leq Cr^s,$$
	for all $0<r<\frac{1}{2}$ and $q_{(z,t_0)}$ from Proposition \ref{boundaryHarnackExpansion}. Let $(x,t)$ be the point closer to the boundary than $(x,t')$, $(z,t)$ be its closest boundary point at time $t$, and let $(z',t')$ be the closest to $(x,t')$.  We can estimate
	\begin{align*}
		|u(x,t)-u(x,t')|\leq & |u(x,t)-q_{(z,t)}d^s(x,t) - u(x,t') + q_{(z',t')}d^s(x,t')|\\
		&+|q_{(z,t)}d^s(x,t)-q_{(z',t')}d^s(x,t')|\\
		\leq & \text{I} +\text{II}.
	\end{align*}
	To estimate $ \text{II}$, we observe that $q_{(z,t)} = \frac{u}{d^s}(z,t)$ which is a $C^{2s-1}_p$ function, and hence 
	$|q_{(z,t)} - q_{(z',t')}|\leq C(|z-z'|^{2s-1} + |t-t'|^{\frac{2s-1}{2s}}).$ Moreover by triangle inequality and the assumption that $|t-t'|$ is big, we have that $|z-z'|+|t-t'|^{\frac{1}{2s}}\leq C |t-t'|^{\frac{1}{2s}}.$ Hence
	\begin{align*}
		\text{II}\leq & |q_{(z,t)}-q_{(z',t')}||d^s(x,t)| + |q_{(z',t')}||d^s(x,t)-d^s(x,t')|\\
		\leq &C|t-t'|^{\frac{2s-1}{2s}} |t-t'|^{\frac{1}{2}} + C |t-t'|^s\\
		\leq & C|t-t'|^s.
	\end{align*}
	Let us turn to the term $\text{I}$. We fix $x_0$ on the line going through $z$ and $x$, far enough from the boundary (to be specified later) and define $x_k = (1-2^{-k})x + 2^{-k}x_0$, for $k\in\N.$ We denote $v_z(\cdot) = u(\xi,\tau)- q_{(z,t)}d^s(\cdot)$ and  compute
	\begin{align*}
		\text{II} =& |v_z(x,t)-v_{z'}(x,t')|\\
		\leq & |v_z(x_0,t)-v_{z'}(x_o,t')|+ \sum_{k=1}^\infty |v_z(x_k,t)-v_z(x_{k-1},t)| +\sum_{k=1}^\infty |v_{z'}(x_k,t')-v_{z'}(x_{k-1},t')|  \\
		\leq &|v_z(x_0,t)-v_z(x_0,t')|+|q_{(z,t)}d^s(x_0,t') - q_{(z',t')}d^s(x_0,t')|\\
		&+  \sum_{k=1}^\infty C|x_{0}-z|^{s} |x_k-x_{k-1}|^{2s-1}+\sum_{k=1}^\infty C|x_{0}-z'|^{s} |x_k-x_{k-1}|^{2s-1}\\
		\leq & C|t-t'|^{\frac{3s-1}{2s}} + C|t-t|^{\frac{1}{2}}\sum_{k=1}^\infty 2^{-k(2s-1)}|x-x_0|^{2s-1}\\
		\leq & C|t-t'|^s,
	\end{align*}
	if $x_0$ is chosen so that $|x_0-z|\leq C_0|t-t'|^{\frac{1}{2s}}$ and that $(x_0,t),(x_0,t')\in Q_r(y_0,t_0)\subset Q_{2r}(y_0,t_0)$ for some $(y_0,t_0)\in \Omega.$ Hence also 
	$$\left[ u\right]_{C_t^{s}(\Omega\cap Q_{1/2})}\leq C,$$
	and the claim is proven.
\end{proof}

	\section{Regularity of the free boundary in space and time} \label{sec:obstProb}
	
	In this section we connect the established boundary Harnack results to the regularity of the free boundary in the parabolic nonlocal obstacle problem in the subcritical regime. 
	We say that a function $u$ solves the parabolic obstacle problem for integro-differential operators, if it holds
	\begin{equation}\label{eq:theObstacleProblem}
		\begin{array}{rcll}
		\min\{(\partial_t+L)u,u-\varphi \}&=&0\quad &\text{ in }\R^n\times(-1,1)\\
		u(\cdot,-1)&=&\varphi\quad &\text{ in }\R^n,
		\end{array}
	\end{equation}
	for a given function $\varphi\colon\R^n\to\R$ called the obstacle and some non-local elliptic operator $L$ of the form \eqref{eq:operatorForm}. 
	The $C^{1,\alpha}$ regularity of the free boundary near regular points has been established in \cite{BFR18,FRS22}. Here we prove that the free boundary is $C^{2,\alpha}_p$.

	\begin{proposition}\label{prop:C2alphaRegOfTheFB}
		Let $s\in(\frac{1}{2},1)$ and set  $\alpha=\min\{s,2-2s\}$. Suppose that $u$ solves \eqref{eq:theObstacleProblem} for some operator $L$  of the form \eqref{eq:operatorForm}, with kernel $K\in C^5(\S^{n-1})$. Assume that the obstacle $\varphi\in C^{4}(\R^n).$ Suppose that $0\in\partial\{u>\varphi\}$ is a regular free boundary point.

		Then the normal to the free boundary is  $C^{s+\alpha-\varepsilon}_p$ in $Q_{r_0},$ for some $r_0>0$ and every $\varepsilon>0$.
	\end{proposition}
	\begin{proof}
		In \cite{BFR18} (see also \cite{FRS22}) they prove that there exists $r_0>0$ so that in $Q_{r_0}$ the free boundary is $C^{1,\gamma}$ for some $\gamma>0$ in space-time and that up to a rotation of coordinates $\partial_n u\geq c_1d^s$, for some $c_1>0.$
		We can assume that $r_0=1$, since the problem is scale invariant. Let us denote 
		$$v = u-\varphi.$$
		Since $u$ solves \eqref{eq:theObstacleProblem}, the partial derivatives of $v$ solve 
		$$\left\lbrace \begin{array}{rcll}
		(\partial_t + L) w &=&\partial_e f&\text{in }\Omega\cap Q_1\\
		w&=&0&\text{in }\Omega^c\cap Q_1,
		\end{array}\right.$$
		for $f=-L\varphi$ and $\Omega=\{u>\varphi\} = \{v>0\}.$
		
		Remember that $f$ is independent of time and by assumption on $\varphi$, we have $f\in C^{1+s}(\R^n)$ and hence $\partial_ef\in C^{s}_p(\Omega\cap Q_1)$. From \cite[Corollary 1.6]{FRS22} (see also \cite{CF13} for the case $(-\Delta)^s$ and with a smaller $\alpha$) we deduce, that all the partial derivatives of $u$ are $ C^{\frac{\alpha-\varepsilon/2}{2s}}_t(\R^n\times(-1,1))$, and hence the same is true for partial derivatives of $ v$.
		Therefore we can apply Corollary \ref{cor:1-epsRegOfQuotient}, which gives that all quotients $v_i/v_n$ and $v_t/v_n$ are $C^{1-\varepsilon}_p(\overline{\Omega}\cap Q_{1/2}),$ with bounds on the norms. 
		
		Now notice that every component the normal vector $\nu(x,t)$ to the level set $\{v=\tau\}$, $\tau>0$ can be expressed as
		$$\nu^i(x,t) = \frac{\partial_i v}{|\nabla_{(x,t)} v|}(x,t) = \frac{\partial_i v/\partial_nv}{\left(\sum_{j=1}^{n-1}(\partial_jv/\partial_nv)^2 + 1 + (\partial_tv/\partial_nv)^2\right)^{1/2}}.$$ 
		Letting $\tau\downarrow0$, we get that the normal vector is $C^{1-\varepsilon}_p(\partial \Omega\cap Q_{1/2}).$ Hence the free boundary is $C^{2-\varepsilon}_p$ in $Q_{1/2}$.
		
		Now we can apply Corollary \ref{cor:expansion2functions3s}, which gives that the quotients $v_i/v_n$ and $v_t/v_n$ are $C^{s+\alpha-\varepsilon}_p(\overline{\Omega}\cap Q_{1/2}),$ which furthermore implies that the normal to the free boundary is $C^{s+\alpha-\varepsilon}_p$ in $Q_{1/2}$.
	\end{proof}
		
	Note that the above proof works already if the initial regularity of the free boundary is~$C^1_p$.
	
	Theorem \ref{thm:1.1} follows.

	\begin{proof}[Proof of Theorem \ref{thm:1.1}]
		The claim follows from Proposition \ref{prop:C2alphaRegOfTheFB}, noting that $s+\alpha>1$.
	\end{proof}

\section{Operator evaluation of $d^s$} \label{sec:Lds}

	This last section we devote to the analysis of the evaluation of operators satisfying \eqref{eq:operatorForm} of the distance function. It is put at the end due to its technical and computational nature. The regularity in space has been studied in \cite{AR20,K21b}. We follow their approach, but we focus on the regularity in time. In fact the regularity in time behaves as expected: the regularity in space is proved to be $C^{\beta-1-s}$, if the domain is of class $C^\beta$, while  we establish the $C^\frac{\beta-1-s}{2s}$ regularity in time. Precisely we prove the following.

\begin{proposition}\label{prop:Lds}
	Let $\Omega$ be $C^\beta_p$ in $Q_1$, for some $\beta\in(1+s,2)$. Assume that the operator $L$ satisfies \eqref{eq:operatorForm} with its kernel $K\in C^{2\beta+1}(\S^{n-1}).$ If $Q\colon \R^n \to \R$ is a  polynomial of degree one with gradient bounded by $1$, we have
	\begin{align*}
	\left[ L(Qd^s_+)\right]_{C^{\beta-1-s}_p(Q_r\cap\Omega)} &\leq C(|Q(0)|+r^{s-\varepsilon}), 		
	\end{align*}
	for every $r\leq \frac{1}{2},$ and every $\varepsilon\in(0,s)$.
	The constant $C$ depends only on $n,s,G_0,\|K\|_{C^\alpha{2\beta+1}(\S^{n-1})}$ and ellipticity constants.
	
	Furthermore, when $(x_0,t_0)\in \Omega\cap Q_{1/2}$, with $d(x_0,t_0) \leq  C_0 r$, so that $Q_{2r}(x_0,t_0)\subset \Omega$, we have 
	$$	\left[ L(Qd^s_+)\right]_{C^{\beta-1}_p(Q_r(x_0,t_0))} \leq C(|Q(0)|+|x_0|)r^{-s}.$$
	
	Moreover if $Q(0)=0$, whenever $d_x(x_0) = |x_0| = C_0r,$ with a suitable constant $C_0$ depending only on $n$, it holds
	$$\left[ L(Qd^s_+)\right]_{C^{\beta-1-\varepsilon}_p(Q_r(x_0,0))}\leq C.$$
\end{proposition}

\begin{proof}
	The regularity of $L(Qd^s)$ in space follows from \cite[Corollary 2.3]{AR20}, since the generalized distance function  satisfies the suitable estimates (see Definition \ref{def:generalisedDistance}). Let us show the regularity in time.
	
	We start with applying \cite[Lemma 2.4]{AR20}, to get
	$$L(Qd_+^s)(x,t) = -\frac{1}{2s}p.v.\int_{\R^n} (Q(y)sd_+^{s-1}(y,t)\nabla d(y,t) + \nabla Q d^s(y,t))\cdot (y-x)K(y-x)dy.$$
	We now introduce the new variable $y = \phi_t(\eta)$, where $\phi_t$ comes as the inverse of the space component of the flattening map $\psi$: $\psi(x,t) = (\psi_t(x),t),$ $\phi_t= \psi_t^{-1}.$ Note that then $d(\phi_t(\eta),t) = \eta_n$ and that $\phi_t\in C^{\beta}_p\cap C^2_x(\{x_n\neq 0\})$ and satisfies 
	$ |D^2\phi_t (x)|\leq C |x_n|^{\beta-2}.$
	We get
	\begin{align*}
	L(Qd_+^s)(x,t) = &-\frac{1}{2s}p.v.\int_{\R^n} (\eta_n)_+^{s-1}\rho(\eta,t) (\phi_t(\eta)-x)K(\phi_t(\eta)-x)d\eta\\
	=&  -\frac{1}{2s}p.v.\int_{B_1} (\eta_n)_+^{s-1}\rho(\eta,t) (\phi_t(\eta)-x)K(\phi_t(\eta)-x)d\eta\\
	&  -\frac{1}{2s}\int_{B_1^c} (\eta_n)_+^{s-1}\rho(\eta,t) (\phi_t(\eta)-x)K(\phi_t(\eta)-x)d\eta
	\end{align*}
	where we denoted $\rho(\eta,t) = (Q(\eta)\nabla d(\phi(\eta,t)) + \nabla Q d(\phi(\eta,t) )|\nabla \phi_t(\eta)|.$ Note that $\phi_t\in C^\beta_p$ and $\rho \in C^{\beta-1}_p$, hence the integral in $B^c_1$ is  $C^{\frac{\beta-1}{2s}}_t$ as well. Therefore we only need to study the regularity of the first integral. We denote
	$$\hat{I}(x,t) = -\frac{1}{2s}p.v.\int_{B_1} (y_n)_+^{s-1}\rho(y,t) (\phi_t(y)-x)K(\phi_t(y)-x)dy,$$
	and
	$$I(x,t) = \hat{I}(\psi^{-1}(x,t)) = -\frac{1}{2s}p.v.\int_{B_1} (y_n)_+^{s-1}\rho(y,t) (\phi_t(y)-\phi_t(x))K(\phi_t(y)-\phi_t(x))dy.$$
	Since $\psi^{-1}$ is $C^{\beta}_p$ (and hence Lipschitz), it is sufficient to show that $I$ satisfies the stated estimate. To do so we expand $\phi_t(x)-\phi_t(y) = D\phi_t(x) (x-y) + S(x,y,t).$  Similarly		 
	\begin{align*}
	K(\phi_t(x)-\phi_t(y)) &= K(D\phi_t(x) (x-y) + S(x,y,t))\\
	& = |x-y|^{-n-2s}K(D\phi_t(x) \langle x-y\rangle + |x-y|^{-1}S(x,y,t))
	\\&= |x-y|^{-n-2s}\left(K(D\phi_t(x) \langle x-y\rangle) + K_1(x,y,t)\right)
	\end{align*}
	Before plugging the expansions in the formula above, we estimate the newly obtained terms. Using the fundamental theorem of calculus, we can write 
	\begin{align*}
	S(x,y,t) &= \int_x^y \left(D\phi_t(\xi) - D\phi_t(x)\right)(y-x)d\xi\\
	&= \int_x^y \int_x^\xi (\xi-x)^T D^2\phi_t(\eta) d\eta (y-x)d\xi.
	\end{align*}
	Since $\phi_t\in C^{\beta}_p$, we can estimate $|S(x,y,t)|\leq C|x-y|^{\beta}$.
	Analogously,
	$$K_1(x,y,t) = \int_0^1 \nabla K \left( D\phi_t(x)\langle x-y\rangle + \xi \frac{S(x,y,t)}{|x-y|} \right)\frac{S(x,y,t)}{|x-y|} d\xi,$$
	and hence
	$|K_1(x,y,t)|\leq C|x-y|^{\beta-1}.$ %TODO: since \phi_t is close to identity (see construction of d and \phi... if boundary is flat enough (zoom) then it is true), (D\phi_t(x) \langle x-y\rangle is close to being on the sphere, and hence for y \in B_{1/2} the expansion is valid (you do it for y in B_1)!
	Moreover, we want to estimate the incremental differences in time of $S$ and $K_1$. Since $D\phi_t$ is $C^{\frac{\beta-1}{2s}}_t$, we get 
	\begin{equation}\label{eq:incrementS}
	|S(x,y,t) - S(x,y,t')|\leq C |x-y||t-t'|^{\frac{\beta-1}{2s}}.
	\end{equation}
	Moreover we can extract more using the fine estimate from Lemma \ref{lem:generalisedDistanceExsistence} and get 
	\begin{equation}\label{eq:incrementsBetter}
	|S(x,y,t) - S(x,y,t')|\leq Cr^{-1} |x-y|^2|t-t'|^{\frac{\beta-1}{2s}},
	\end{equation}
	when $(y,t')\in Q_r(x,t)$, and $r=\frac{x_n}{2}.$.
	Since $\nabla K$ is Lipschitz, increments of $K_1$ inherit the estimates of $S$, and hence
	\begin{equation}\label{eq:incrementK_1}
	|K_1(x,y,t)-K_1(x,y,t')|\leq C |t-t'|^{\frac{\beta-1}{2s}}
	\end{equation}
	and 
	\begin{equation}\label{eq:incrementK1Better}
	|K_1(x,y,t) - K_1(x,y,t')|\leq Cr^{-1} |x-y||t-t'|^{\frac{\beta-1}{2s}},
	\end{equation}
	whenever $(y,t')\in Q_r(x,t)$, and $r=\frac{x_n}{2}.$
	
	We plug the expansions into the expression for $I$ and get\footnote{We omit the principal value symbol in some of the integrals until the end of the proof.}
	\begin{align*}
	\begin{split}
	-2sI(x,t) = & \int_{B_1} (y_n)_+^{s-1}\rho(y,t)D\phi_t(x)(y-x)K(D\phi_t(x)\langle y-x\rangle) |x-y|^{-n-2s}dy\\
	& + \int_{B_1} (y_n)_+^{s-1}\rho(y,t)S(x,y,t)K(D\phi_t(x)\langle y-x\rangle) |x-y|^{-n-2s}dy\\
	& + \int_{B_1} (y_n)_+^{s-1}\rho(y,t)D\phi_t(x)(y-x)K_1(x,y,t) |x-y|^{-n-2s}dy\\
	& + \int_{B_1} (y_n)_+^{s-1}\rho(y,t)S(x,y,t)K_1(x,y,t) |x-y|^{-n-2s}dy\\
	=& I_1+I_2+I_3+I_4.
	\end{split}
	\end{align*}
	We start with analysing $I_1$. We split it furthermore as follows
	\begin{align*}
	I_1 =& \rho(x,t)\int_{B_1} (y_n)_+^{s-1}D\phi_t(x)(y-x)K(D\phi_t(x)\langle y-x\rangle) |x-y|^{-n-2s}dy \\
	&+ \int_{B_1} (y_n)_+^{s-1}(\rho(y,t)-\rho(x,t))D\phi_t(x)(y-x)K(D\phi_t(x)\langle y-x\rangle) |x-y|^{-n-2s}dy\\
	=& I_{11}+ I_{12}.
	\end{align*}
	The integral $I_{11}$ is the term with the lowest order of $|x-y|$, but we exploit the fact that it is almost an evaluation of some homogeneous operator of order $2s$ of the function $(x_n)_+^p$. To do so, we decouple the space variables in the following way
	$$I_{11} = I_{11}(x,x,t)$$
	where
	$$I_{11}(\xi,x,t) = \rho(x,t)\int_{B_1} (y_n)_+^{s-1}D\phi_t(\xi)\langle y-x\rangle K(D\phi_t(\xi)\langle y-x\rangle) |x-y|^{-n-2s+1}dy.$$
	Then by \cite[Lemma 3.4]{K21a} we conclude that
	\begin{align*}
	I_{11}(\xi,x,t) =\rho(x,t)\int_{B_1^c} (y_n)_+^{s-1}D\phi_t(\xi)\langle y-x\rangle K(D\phi_t(\xi)\langle y-x\rangle) |x-y|^{-n-2s+1}dy.
	\end{align*}
	The obtained integral converges and hence $I_{11}$ is $C^{\frac{\beta-1}{2s}}_t$. 
	
	We proceed with analysing $I_{12}$. We want to bound the time incremental difference of $I_{12}$. Since we are getting a lot of terms, we simplify the notation in the following way: we denote $\Delta_t f = f(t)-f(t')$, as well as $\Delta_x f = f(x)-f(y)$, furthermore $\rho = (Q\nabla d + d\nabla Q)J$, where $J$ stands for $|\phi_t(y)|.$ 
	We furthermore split $I_{12}$ 
	$$I_{12} = \int_{B_r(x)} \ldots dy + \int_{B_1\backslash B_r(x)} \ldots dy =B_{12r} + I_{12rc},$$
	for $r=\frac{x_n}{2}.$
	In the region near the pole we need to extract both $|t-t'|^{\frac{\beta-1}{2s}}$, as well as $|x-y|$ from the increment of $\rho$. This is done with adding an subtracting $D\rho(x,t) (y-x)$ from $\rho(y,t)-\rho(x,t)$. We get
	\begin{align*}
	I_{12r} = &\int_{B_{r}(x)} (y_n)_+^{s-1}(\rho(y,t)-\rho(x,t)- D\rho(x,t)(y-x))D\phi_t(x)(y-x)K(D\phi_t(x)( y-x)) dy\\ 
	&+ \int_{B_{r}(x)} (y_n)_+^{s-1} D\rho(x,t)(x-y)D\phi_t(x)(y-x)K(D\phi_t(x)\langle y-x\rangle) |x-y|^{-n-2s}dy.
	\end{align*}
	We can write
	$$ S_1(x,y,t)=\rho(y,t)-\rho(x,t)- D\rho(x,t)(y-x) = \int_x^y(D\rho(\xi,t)-D\rho(x,t))d\xi\cdot (y-x).$$
	Using the simplified notation we compute
	$$D\rho = D( J) (Q\nabla d + d\nabla Q) + J(2\nabla Q \nabla d + QD^2d),$$
	and hence
	\begin{align*}
	\Delta_t D\rho = &  \Delta_t(D( J)) (Q\nabla d + d\nabla Q) + \Delta_tJ(2\nabla Q \nabla d + QD^2d)\\
	&+ D( J) (Q\Delta_t\nabla d + \Delta_td\nabla Q) + J(2\nabla Q \Delta_t \nabla d + Q\Delta_tD^2d).
	\end{align*}
	Lemma \ref{lem:generalisedDistanceExsistence} gives that 
	$$|\Delta_t D(J)|, |\Delta_t D^2d|\leq Cr^{-1}|t-t'|^{\frac{\beta-1}{2s}},\quad\quad |D(J)|,|D^2d|\leq Cr^{\beta-2}, $$
	and so we can bound 
	$$|\Delta_tD\rho(\xi,t)| \leq  Cr^{-1}(Q(0) +|x|+r)|t-t'|^{\frac{\beta-1}{2s}} ,$$
	since $|Q(\xi)|\leq |Q(0)| + C|\xi|\leq C(|Q(0) + |x| + r)$.
	
	We are now ready to estimate the incremental difference of $I_{12r}$. Whenever the difference operator does not apply on terms with $\rho$, we get $|t-t'|^{\frac{\beta-1}{2s}}$ from the other terms, and we are left with 
	$$Cr^{\beta-2}\int_{B_{r}(x)} (y_n)_+^{s-1}| y-x|^{-n-2s+2} dy,$$
	which we can bound with 
	$Cr^{\beta-s-1}$. When the difference operator lands on the term with $\rho,$ we get the following 
	$$Cr^{-1}(|Q(0)|+|x| + r)|t-t'|^{\frac{\beta-1}{2s}}\int_{B_{r}(x)} (y_n)_+^{s-1}| y-x|^{-n-2s+2} dy,$$
	which is bounded by $Cr^{-s}(|Q(0)|+|x| + r) |t-t'|^{\frac{\beta-1}{2s}}$.
	
	We proceed with $I_{12rc}$. We can estimate		
	$$|\Delta_t I_{12rc}|\leq \int_{B_1\backslash B_r(x)} y_{n+}^{s-1}|\Delta_t\Delta_x \rho| |x-y|^{-n-2s+1}dy + C|t-t'|^{\frac{\beta-1}{2s}}\int_{B_1}y_{n+}^{s-1}|x-y|^{-n-2s+\beta}, $$
	where the second term bounds all the other contributions of $\Delta_t$ that do not come from $\Delta_x\rho$. We continue with computing the double incremental difference
	\begin{align*}
	\Delta_t\Delta_x \rho\leq & \Delta_t \left( \Delta_x J (Q\nabla d + d\nabla Q) + J(\Delta_x Q \nabla d + Q\Delta_x\nabla d + \Delta_x d \nabla Q) \right) \\
	\leq & \Delta_t\Delta_x J (Q\nabla d + d\nabla Q) +\Delta_t J(\Delta_x Q \nabla d + Q\Delta_x\nabla d + \Delta_x d \nabla Q)\\
	&+ \Delta_x J (Q \Delta_t\nabla d + \Delta_td\nabla Q) + J(\Delta_x Q \Delta_t\nabla d + Q\Delta_t\Delta_x\nabla d + \Delta_t\Delta_x d \nabla Q).
	\end{align*}
	Where we have the single increments we do the straightforward estimate. It holds  $|\Delta_t\Delta_x J|\leq C |t-t'|^{\frac{\beta-1}{2s}}$. The double increment of the distance function is better, since $d\in C^\beta_p$, so using the fundamental theorem of calculus we can extract $|\Delta_t\Delta_x d|\leq |x-y||t-t'|^{\frac{\beta-1}{2s}}$. 
	Hence we get
	\begin{align*}
	|\Delta_t\Delta_x \rho|\leq &C (|Q(0)|+|x|)|t-t'|^{\frac{\beta-1}{2s}} + C|x-y||t-t'|^{\frac{\beta-1}{2s}} + C|x-y|^{\beta-1}|t-t'|. 
	\end{align*}
	Plugging it inside the integral we end up with 
	$$|\Delta_t I_{12rc}|\leq C\left((|Q(0)|+|x|)|t-t'|^{\frac{\beta-1}{2s}} r^{-s} + |t-t'|^{\frac{\beta-1}{2s}} r^{-\varepsilon} \right) $$
	in view of \cite[Lemma A.9]{AR20} and Lemma \ref{lem:A9fake}.
	
	The other terms $I_2,I_3,I_4$ are estimated in a similar manner. Let us analyse the term $I_2$. Taking the incremental difference and  estimating we get
	$$ |\Delta_t I_2|\leq \int_{B_1} y_{n+}^{s-1}|\rho(y)||\Delta_t S| |x-y|^{-n-2s}dy + C|t-t'|^{\frac{\beta-1}{2s}}\int_{B_1}y_{n+}^{s-1}|x-y|^{-n-2s+\beta}. $$
	The second integral gives $ Cx_n^{-\varepsilon} |t-t'|^{\frac{\beta-1}{2s}}$. In the first one we split $|\rho(y)| \leq C(|Q(0)|+|x|+|x-y|)$. In the term with $|x-y|$ we can use estimate \eqref{eq:incrementS} for $|\Delta_tS|,$ to get 
	$$C|t-t'|^{\frac{\beta-1}{2s}}\int_{B_1} y_{n+}^{s-1}|x-y|^{-n-2s+2}dy,$$
	which is bounded with $Cx_n^{-\varepsilon} |t-t'|^{\frac{\beta-1}{2s}}$, thanks to Lemma \ref{lem:A9fake}.
	The remaining term 
	$$(|Q(0)|+|x|)\int_{B_1} y_{n+}^{s-1}|\Delta_t S| |x-y|^{-n-2s}dy $$
	we split into the integrals in regions $B_r(x)$ and $B_1\backslash B_r(x)$. Away from the pole we can still use the estimate \eqref{eq:incrementS} and estimate the obtained integral with \cite[Lemma A.9]{AR20}, to get $C(|Q(0)|+|x|)|t-t'|^{\frac{\beta-1}{2s}} r^{-s}$. In the region near the pole we use \eqref{eq:incrementsBetter}, which gives 
	$$(|Q(0)|+|x|)r^{-1}|t-t'|^{\frac{\beta-1}{2s}}\int_{B_1} y_{n+}^{s-1} |x-y|^{-n-2s+2}dy $$
	which is also bounded by $C(|Q(0)|+|x|)|t-t'|^{\frac{\beta-1}{2s}} r^{-s}$.
	
	The analysis of the integral $I_3$ is the same as the one of $I_2$, using \eqref{eq:incrementK_1} and \eqref{eq:incrementK1Better} instead of \eqref{eq:incrementS} and \eqref{eq:incrementsBetter}. The term $I_4$ is the same.
	
	We obtained
	$$|I(x,t)-I(x,t')|\leq C \left( (|Q(0)|+|x|)r^{-s} + r^{-\varepsilon} \right) |t-t'|^{\frac{\beta-1}{2s}}, $$
	whenever $(x,t),(x,t')\in Q_r(x_0,t_0)$ with $(x_0)_n = 2r$. This implies the second and the third estimate in the claim, as well as 
	\begin{align*}
	|I(x,t)-I(x,t')|&\leq C \left( |Q(0)|+|x| + r^{s-\varepsilon} \right) |t-t'|^{\frac{\beta-1-s}{2s}}\\
	&\leq C(|Q(0)|+r_0^{s-\varepsilon})|t-t'|^{\frac{\beta-1-s}{2s}},
	\end{align*}
	if additionally $|x_0|\leq r_0$.
	It follows from \cite[Lemma B.2]{K21b}, that $\left[ I\right]_{C^{\beta-1-s}_p(Q_{r_0})}\leq C(|Q(0)|+r_0^{s-\varepsilon}),$ since the space part of the estimate is provided by \cite[Corollary 2.3]{AR20}.
	Since the flattening map $\phi$ is Lipschitz the claim is proved.
\end{proof}

	\appendix
	\section{Technical tools and lemmas} \label{sec:appendix}
	
	%TODO: comment \beta<2s
	\begin{lemma}\label{lem:generalisedDistanceExsistence}
		Let $\Omega$ be $C^{\beta}_p$ in $Q_1$, for some $\beta\in (2s,2).$ 
		Then there exists a function $d\colon \R^{n+1}\to\R$  satisfying
		$$d\in C^{\beta}_p(\R^{n+1})\cap C^\infty_x(\{d>0\}),\quad C^{-1}\operatorname{dist}(\cdot,\partial\Omega\cap Q_1)\leq d\leq C \operatorname{dist}(\cdot,\partial\Omega\cap Q_1) \text{ in }\Omega\cap Q_1,$$
		$$|(\partial_t)^lD^kd|\leq C_{k,l}d^{\beta-k-2sl},\quad \text{in }\{d>0\},\text{ whenever }2sl+k>\beta,$$
		$$\left[ D^2d \right]_{C^{\frac{\beta-1}{2s}}_t(Q_r(x_0,t_0))}\leq Cr^{-1},$$
		whenever $d(x_0,t_0)\leq C_0 r$ and $Q_{2r}(x_0,t_0)\subset \Omega$.
		
		Furthermore there exists a diffeomorphism $\psi\in C^\beta_p(\R^{n+1},\R^{n+1})\cap C^\infty_x(\{d>0\},\{x_n>0\})$, such that $$\psi(\Omega\cap Q_1) = Q_1\cap\{x_n>0\},\quad d(\psi^{-1}(x,t)) = x_n,\quad|D^k \psi|\leq C d^{\beta-k},$$
		in $\{d>0\}, $ for all $k>\beta$ and whenever $d(x_0,t_0)\leq C_0 r$ and $Q_{2r}(x_0,t_0)\subset \Omega$, we have 
		$$\left[ D^2\psi \right]_{C^{\frac{\beta-1}{2s}}_t(Q_r(x_0,t_0))}\leq Cr^{-1}.$$
	\end{lemma}
	\begin{proof}
		Let $f\in C^{\beta}_p(\{|x'|\leq1\}\times(-1,1))$ be such that $\Omega\cap Q_1 = \{x_n > f(x',t)\}.$ We can extend $f$ to full space $\R^{n-1}\times\R$, so that its norm does not increase. Let $L= 8||\nabla  f||_{L^\infty}$ and $K = 16s 8^{2s-1} ||\partial_tf||_{L^\infty} $. 
		Choose $\varphi\in C^\infty_c(B_1)$ with $\int\varphi dz = 1$, $\psi\in C^{\infty}_c(-1,1)$ with $\int_{-1}^1\psi d\sigma = 1$, and define 
		$$F(x,t,\tau) = x_n - \int_{\R^{n+1}}f\left(x'-\frac{\tau}{L}z',t - \frac{\tau^{2s}}{K}\sigma\right)\varphi(z)\psi(\sigma)dzd\sigma.$$
		Since $\varphi$ and $\psi$ are compactly supported, $F\in C^\beta_p(\R^{n+1})\cap C^{2s}_\tau$. We compute
		\begin{align}\label{eq:firstDerivatives}
		\begin{split}
		\partial_\tau F(x,t,\tau) = & \frac{1}{L}\int_{\R^{n+1}}\nabla f \left(x-\frac{\tau}{L}z,t- \frac{\tau^{2s}}{K}\sigma\right) \cdot z\varphi(z)\psi(\sigma)dzd\sigma +\\
		& + \frac{2s}{K}\tau^{2s-1}\int_{\R^{n+1}}\partial_t f \left(x-\frac{\tau}{L}z,t- \frac{\tau^{2s}}{K}\sigma\right) \sigma\varphi(z)\psi(\sigma)dzd\sigma,  \\
		\nabla F  =& e_n - \int_{\R^{n+1}}\nabla f\left(x-\frac{\tau}{L}z,t- \frac{\tau^{2s}}{K}\sigma\right)\varphi(z)\psi(\sigma)dzd\sigma,\\
		\partial_t F=&  -\int_{\R^{n+1}}\partial_t f\left(x-\frac{\tau}{L}z,t- \frac{\tau^{2s}}{K}\sigma\right)\varphi(z)\psi(\sigma)dzd\sigma,
		\end{split}
		\end{align}
		so thanks to the choice of $L$ and $K$, we have
		$|\partial_\tau F(x,t,\tau)|\leq \frac{1}{4},$ for $\tau<8$. Therefore, in combination with $ F(x,t,0) = x_n-f(x',t)$, we conclude that for every $(x,t)\in\{ x_n - f(x,t)\in(0,3) \}$ there exists a unique $d = d(x,t)\in (0,8),$ so that 
		\begin{equation}\label{eq:fixedPoint}
		d(x,t) = F(x,t,d(x,t)),
		\end{equation}
		as well as 
		$$\left| x_n-f(x,t) - d(x,t)\right| = \left| F(x,t,0) - F(x,t,d(x,t))\right| \leq \frac{1}{4}d(x,t),$$
		which implies that $d$ is comparable to $x_n-f(x,t)$.
		By implicit function theorem $d$ is $ C^1$ in both variables. Differentiating \eqref{eq:fixedPoint}, we get
		$$\nabla d(x,t) = \frac{\nabla F(x,t,d(x,t))}{1-\partial_\tau F(x,t,d(x,t))}, \quad \partial_t d(x,t) = \frac{\partial_t F(x,t,d(x,t))}{1-\partial_\tau F(x,t,d(x,t))}.$$
		Notice also, that $\partial_n d \geq \frac{1}{2}.$
		Moreover since $F\in C^\infty(\{\tau>0\})$ (it is a convolution with a $C^\infty$ function), it also holds $d\in C^\infty(\{d>0\})$. To get the expressions for higher order derivatives of $d$, we differentiate \eqref{eq:fixedPoint} and then insert the suitable derivatives of $F$. 
		We show here only the computations for the second order derivatives, higher order ones are established analogously. Procedure is always the following. We perform an affine change of variables, to move the variables $x,t$ to $\varphi,\psi$, then derive and finally change the coordinates back. The computations are long and technical, therefore we present only the results. We get
		\begin{align*}
		\partial_j\partial_i F &= -\frac{L}{\tau} \int_{\R^{n+1}}\partial_if\left(x-\frac{\tau}{L}z,t- \frac{\tau^{2s}}{K}\sigma\right)\partial_j\varphi(z) \psi(\sigma)dzd\sigma\\
		&=-\frac{L}{\tau} \int_{\R^{n+1}}\left(\partial_if\left(x-\frac{\tau}{L}z,t- \frac{\tau^{2s}}{K}\sigma\right) - \partial_if(x,t)\right)\partial_j\varphi(z) \psi(\sigma)dzd\sigma
		\end{align*}
		where in the last equality we used that $\int\partial_j\varphi = 0$, since $\varphi$ is compactly supported. Similarly we get
		\begin{align*}
		\partial_j\partial_\tau F =& \frac{1}{\tau} \int_{\R^{n+1}}\nabla f\left(x-\frac{\tau}{L}z,t- \frac{\tau^{2s}}{K}\sigma\right) \cdot \partial_j(z\varphi(z)) \psi(\sigma)dzd\sigma\\
		&+2s\frac{L}{K}\tau^{2s-2} \int_{\R^{n+1}}\partial_t f\left(x-\frac{\tau}{L}z,t- \frac{\tau^{2s}}{K}\sigma\right) \partial_j\varphi(z) \psi(\sigma)dzd\sigma\\
		\partial_\tau\partial_\tau F = & -\frac{1}{\tau L} \int_{\R^{n+1}} \sum_{j,k} \partial_kf \left(x-\frac{\tau}{L}z,t- \frac{\tau^{2s}}{K}\sigma\right)  \partial_j(z_kz_j\varphi(z)) \psi(\sigma) dzd\sigma\\
		& -\frac{2s}{\tau L} \int_{\R^{n+1}} \sum_{k} \partial_kf \left(x-\frac{\tau}{L}z,t- \frac{\tau^{2s}}{K}\sigma\right) z_k \varphi(z) (\sigma\psi(\sigma))' dzd\sigma\\
		& - \frac{2s}{K}\tau^{2s-2} \int_{\R^{n+1}}  \partial_tf \left(x-\frac{\tau}{L}z,t- \frac{\tau^{2s}}{K}\sigma\right) \sigma \nabla\cdot (z\varphi(z)) \psi(\sigma) dzd\sigma\\
		& -\frac{(2s)^2}{K} \tau^{2s-2}\int_{\R^{n+1}} \partial_tf \left(x-\frac{\tau}{L}z,t- \frac{\tau^{2s}}{K}\sigma\right) \sigma^{2s-1} \varphi(z) (\sigma^{\frac{1+2s}{2s}}\psi(\sigma))' dzd\sigma.
		\end{align*}
		Using that $f\in C^{\beta}_p$, we conclude that all the derivatives above are bounded with $C\tau^{\beta-2}$.
		We also compute
		\begin{align*}
		\partial_{tt} F = & -\frac{K}{\tau^{2s}}\int_{\R^{n+1}} \partial_t f \left(x-\frac{\tau}{L}z,t- \frac{\tau^{2s}}{K}\sigma\right) \varphi(z) \psi'(\sigma) dz d\sigma\\
		\nabla\partial_tF = & -\frac{L}{\tau}\int_{\R^{n+1}} \partial_t f \left(x-\frac{\tau}{L}z,t- \frac{\tau^{2s}}{K}\sigma\right) \nabla\varphi(z) \psi(\sigma) dz d\sigma\\
		\partial_\tau\partial_t F = & \frac{1}{\tau}\int_{\R^{n+1}}\partial_t f\left(x-\frac{\tau}{L}z,t- \frac{\tau^{2s}}{K}\sigma\right) \nabla\cdot(z\varphi(z)) \psi(\sigma)dzd\sigma\\
		&+\frac{2s}{\tau} \int_{\R^{n+1}}\partial_t f\left(x-\frac{\tau}{L}z,t- \frac{\tau^{2s}}{K}\sigma\right) \varphi(z) (\sigma\psi(\sigma))'dzd\sigma.
		\end{align*}
		Since again in every expression there are derivatives of compactly supported functions, we deduce 
		$|\partial_{tt}F|\leq C\tau^{\beta-4s}$ and $|\nabla\partial_{t}F|,|\partial_\tau\partial_{t}F|\leq C\tau^{\beta-2s-1}$. 		
		
		Finally, differentiating \eqref{eq:fixedPoint} twice, we extract
		\begin{align*}
		D^2 d & = \frac{1}{1-\partial_\tau F}\left(D^2 F + 2\nabla\partial_\tau F \nabla d^T + \partial_{\tau\tau} F \nabla d \nabla d^T\right)\\
		\partial_{tt} d & = \frac{1}{1-\partial_\tau F}\left(\partial_{tt} F + 2\partial_t\partial_\tau F \partial_t d + \partial_{\tau\tau} F (\partial_t d)^2\right)\\
		\nabla\partial_t d & = \frac{1}{1-\partial_\tau F}\left(\nabla\partial_t F + \nabla\partial_\tau F \partial_t d + \partial_t\partial_\tau F \nabla d+ \partial_{\tau\tau} F \nabla d \partial d\right),
		\end{align*}
		where all the arguments are either $(x,t)$ or $(x,t,d(x,t))$. The previous computations yield
		$$|D^2 d|\leq Cd^{\beta-2},\quad\quad |\partial_{tt}d|\leq Cd^{\beta-4s},\quad\quad|\nabla\partial_t d|\leq C d^{\beta-1-2s}.$$
		These estimates imply that $d\in C^\beta_p(\{d\geq 0 \}).$ Finally using the estimates for $D^3 d$ and $\partial_t D^2 d$, we get 
		$$\left[ D^2d\right]_{C^{\beta-1}_p(Q_r(x_0,t_0))}\leq Cr^{-1},$$
		whenever $d(x_0,t_0)\leq C_0 r$ and $Q_{2r}(x_0,t_0)\subset \Omega$.
		To extend $d$ to the full set $\{x_n>f(x,t)\}$, we take a cut-off $\phi\in C^\infty_c(\left[ 0,\infty\right) ),$ such that $\phi=1$ in $\left[ 0,3\right) $,  $\phi =0$ in $\left[ 4,\infty\right) $ and define
		$$\textbf{d}(x,t) = \phi(d(x,t)) d(x,t) + (1-\phi(d(x,t)) x_n,$$
		which satisfies all properties from the claim.
		% note the comparability between d and g = x_n - f. Then d\in(3,4) is mapped between g\in (2,3), and since the domain is not too curved it folows.
		% to get definition for d on the complement, just repea the procedure with g = x_n + f instead, then define d with suitable signs - then check that the derivatives match on the intersection ... (they do)
		% or just put absolute values in \tau^2s
		
		Once we get $d$, we define 
		$$\psi(x,t) = (x',d(x,t),t).$$
		Then $|D_{x,t}\psi| = |\partial_n d| >\frac{3}{4}$, and the derivatives of $\psi$ inherit estimates of derivatives of $d$.
	\end{proof}

	\begin{lemma}\label{lem:divisionLemma}
		Let $s\in(\frac{1}{2},1)$, $a\in(0,2s)$ and $b\in(0,1)$.
		Let  $f,g,Q$ be functions on $Q_r(x_0,t_0)$ which satisfy $||f-Qg||_{L^\infty(Q_r)}\leq Cr^{a+b}$, and $\left[ f-Qg\right] _{C_p^{a+b}(Q_r(2re_n,0))}\leq C$. Assume that $g$ satisfies $||g^{-1}||_{L^\infty(Q_r(2re_n,0))}\leq C r^{-b}$, and $\left[ g\right] _{C_p^{a}(Q_r(2re_n,0))}\leq Cr^{b-a},$ for all $r\in(0,1)$.
		
		Then we have $$\left[ \frac{f}{g}-Q\right] _{C_p^a(Q_r(2re_n,t_0))}\leq C.$$
	\end{lemma}
	
	\begin{proof}
		We denote $v = f-Qg$. Note that by \cite[Lemma B.2]{K21b} $v$ is $C^{\alpha+\beta}_p(\mathcal{C})$, where $\mathcal{C} =  \cup_{r>0} Q_r(2re_n,0)$. Hence also $\nabla v(0,0) = 0$ if $a+b>1$, and then also $||\nabla v||_{Q_r\cap\mathcal{C}}\leq Cr^{a+b-1}$.
		 
		Assume first that $a<1.$ Then 
		\begin{align*}
			\left[ \frac{v}{g}\right]_{C^a_p(Q_r(2re_n,0))} \leq &\left[ v\right] ||g^{-1}|| + ||v||\left[ g^{-1}\right] \\
			\leq & Cr^{b} r^{-b} + Cr^{a+b} r^{-2b} r^{b-a} \\
			\leq & C.
		\end{align*}
		
		Let now $a>1$.
		We estimate 
		\begin{align*}
			\left[ g^{-1}\right]_{C^a_p} = &\left[ g^{-1}\right]_{C^{\frac{a}{2s}}_t} + \left[ \nabla (g^{-1})\right]_{C^{a-1}_p}\\
			\leq & ||g^{-2}|| \left[ g\right]_{C^{\frac{a}{2s}}_t} + 2 \left[ g^{-1}\right]_{C^{a-1}_p} ||g^{-1}\nabla g|| + ||g^{-2}|| \left[ \nabla g\right]_{C^{a-1}_p}\\
			\leq & Cr^{-2b} r^{b-a} + C r^{-2b} r^{b-a+1} r^{-b}r^{b-1} + C r^{-2b}r^{b-a}\\
			\leq & Cr^{-a-b},
		\end{align*}
		where the estimate for $||\nabla g||$ follows from the one for $\left[ \nabla g\right]_{C^{a-1}_p}$. This gives 
		\begin{align*}
		\left[ \frac{v}{g}\right]_{C^a_p(Q_r(2re_n,0))} \leq &\left[ v\right] ||g^{-1}|| + ||v||\left[ g^{-1}\right] +\left[ v\right]_{C^{a-1}_p} ||\nabla (g^{-1})|| + ||\nabla v||\left[ g^{-1}\right]_{C^{a-1}_p}\\
		\leq & Cr^{b} r^{-b} + Cr^{a+b}r^{-a-b} + C r^{b+1} r^{-2b}r^{b-1} + Cr^{a+b-1}r^{-2b}r^{b-a+1}\\
		\leq & C,
		\end{align*}
		which proves the claim.
	\end{proof}
	
	\begin{lemma}\label{lem:interiorRegularityWithGrowth}
		Let $s\in(0,1)$ and let $L$ be an operator of the form \eqref{eq:operatorForm}. Let $u$ be a solution of
		$$(\partial_t+ L)u = f,\quad\quad \text{ in }Q_1.$$

		Then
		$$\left[ u\right]_{C^{2s-\varepsilon}_p(Q_{1/2})}\leq C\left( ||f||_{L^\infty(Q_1)}+ \sup_{R>1}R^{-2s+\varepsilon}|| u||_{L^\infty(B_R\times(-1,1))}  \right),$$
		where $C$ depends only on $n,s,\varepsilon$ and ellipticity constants.
		
		If additionally $f\in C^{\alpha}_p(Q_1)$, for some $\alpha\in(0,1),$ so that $\frac{\alpha}{2s}\in(0,1)$, we have
		$$\left[ u\right]_{C^{\alpha+2s}_p(Q_{1/2})}\leq C\left( \left[ f\right]_{C^\alpha_p(Q_1)}+||u||_{L^\infty(Q_1)} + \sup_{R>1}R^{-2s+\varepsilon}\left[ u\right]_{C^\alpha_p(B_R\times(-1,1))}  \right).$$ 
		
		Moreover if the kernel $K$ of the operator $L$ satisfies 
		$\left[ K\right]_{C^\alpha(B_r^c)}\leq C_1 r^{-n-2s-\alpha}$, $r>0$, then\footnote{Note that the assumption on the kernel is fulfilled if the kernel is homogeneous and $C^\alpha(\S^{n-1})$.}
		\begin{align*}
			\left[ u\right]_{C^{\alpha+2s}_p(Q_{1/2})}\leq &C\Big( \left[ f\right]_{C^\alpha_p(Q_1)}+\sup_{R>1}R^{-2s-\alpha+\varepsilon}||u||_{L^\infty(B_R\times(-1,1))} \\
			&\quad\quad\quad\quad\quad\quad\quad\text{ }+ \sup_{R>1}R^{2s-\varepsilon}\left[ u\right]_{C^{\frac{\alpha}{2s}}_t(B_R\times(-1,1))}  \Big).
		\end{align*}
	\end{lemma}
	\begin{proof}
		The proof the first estimate we take a smooth cut off function $\chi$, applying the estimates on $u\chi$ and computing the error terms. 
		Let $\chi\in C^\infty_c(B_1)$ so that $\chi\equiv 1$ in $B_{5/6}$. Then the function $\bar{u} = u\chi$ solves 
		$$
		(\partial_t  + L)\bar{u} = f+\bar{f} $$
		where $\bar{f}= -L(u(1-\chi))$. We can estimate
		\begin{align*}
		|\bar{f}(x,t)|\leq & \int_{B_{4/5}^c(-x)}|u(1-\chi)|(x+y,t)K(y)dy\\
		\leq & CC_0 \int_{B_{4/5}^c(-x)}(1+|y|)^{2s-\varepsilon}|y|^{-n-2s}dy\\
		\leq CC_0,	
		\end{align*}
		for $C_0= \sup_{R>1}R^{-2s+\varepsilon}||u||_{L^\infty(B_R)\times(-1,1)}$. The first claim now follows from \cite[Theorem 1.3]{FR17}.
		
		For the second claim, we estimate
		\begin{align*}
		\begin{split}
		|\bar{f}(x,t)-\bar{f}(x',t')| =&  \left| \int_{B_{4/5}^c(-x)}u(1-\chi)(x+y,t)K(y)dy\right.\\
		&\left.-\int_{B_{4/5}^c(-x')}u(1-\chi)(x'+y,t')K(y)dy \right|\\
		\leq& \int_{B_{1/20}^c}|u(x+y,t)-u(x'+y,t')|K(y)dy\\
		&+ ||u||_{L^\infty(Q_1)}\int_{B_1\backslash B_{1/20}}|\chi(x+y)-\chi(x'+y)|K(y)dy\\
		\leq&  (|x-x'|^\alpha+|t-t'|^{\frac{\alpha}{2s}})\left(\int_{B_{1/20}^c}\left[u \right]_{C^\alpha(B_{|y|+1})}K(y)dy + C_\chi||u||_{L^\infty(Q_1)}\right)\\
		\leq&  (|x-x'|^\alpha+|t-t'|^{\frac{\alpha}{2s}})\left(\Lambda C_0
		\int_{B_{1/20}^c}\frac{(|y|+1)^{2s-\varepsilon}}{|y|^{n+2s}}dy+C||u||_{L^\infty(Q_1)}\right)\\
		\leq& (|x-x'|^\alpha+|t-t'|^{\frac{\alpha}{2s}})\left(CC_0+C||u||_{L^\infty(Q_1)}\right),
		\end{split}
		\end{align*}
		where $C_0$ stands for $\sup_{R>1} R^{-2s+\varepsilon}  \left[ u_1\right]_{C^{\alpha}_p(B_R\times(-1,1))}.$ In view of \cite[Theorem 1.1]{FR17}, the second claim is proven.
		
		For the moreover case we need to notice that 
		\begin{align}\label{eq:cutoff}
		\begin{split}
			\int_{B_{4/5}^c(-x)}u(1-\chi)(x+y,t)K(y)dy-\int_{B_{4/5}^c(-x')}u(1-\chi)(x'+y,t')K(y)dy = \\
			\int_{B_{4/5}^c}u(1-\chi)(y,t)K(y-x)dy-\int_{B_{4/5}^c}u(1-\chi)(y,t')K(y-x')dy=\\
			\int_{B_{4/5}^c}\left(u(1-\chi)(y,t)-u(1-\chi)(y,t')\right)K(y-x')dy+\text{   } \\+ \int_{B_{4/5}^c}u(1-\chi)(y,t)\left(K(y-x)-K(y-x')\right)dy.\text{   }
		\end{split}
		\end{align}
		The first term can be bounded with 
		$$ C \sup_{R>1} R^{-2s+\varepsilon}\left[ u\right]_{C^{\frac{\alpha}{2s}}_t(B_R\times(-1,1))} |t-t'|^{\frac{\alpha}{2s}},$$
		while the second one with 
		$$ C C_1\sup_{R>1}R^{-2s-\alpha+\varepsilon}||u||_{L^\infty(R_R\times(-1,1))} |x-x'|^{\alpha}.$$
		Using \cite[Theorem 1.1]{FR17} finishes the proof.
	\end{proof}

\begin{lemma}\label{lem:fakeLemma4.5}
	Let $a,\alpha>0$ such that $a+\alpha>1+s$, $\alpha<s$, and $u_1,u_2,d\in C(Q_1)$, such that
	$u_2\in C^s_p(Q_1)$, $u_2\geq c_0 d^s$, for some $c_0>0$ and
	$$\left[ u_1 - q_ru_2 - Q^{(1)}_r d^s\right]_{C^\alpha_p(Q_r)}\leq C_0 r^{a},$$
	for some $q_r\in\R$ and some $1$-homogeneous polynomials $Q_r^{(1)}$. 
	
	Then there exist $q_0$ and $Q_0^{(1)},$ so that 
	$$\left[ u_1 - q_0u_2 - Q^{(1)}_0 d^s\right]_{C^\alpha_p(Q_r)}\leq CC_0 r^{a}.$$
\end{lemma}
\begin{proof}
	Writing $Q^{(1)}_r(x) = Q^{(1)}_r\cdot x$ and using rescaled \cite[Lemma A.10]{AR20}, we compute
	\begin{align*}
	|Q^{(1)}_r-Q^{(1)}_{2r}|\leq& Cr^{-1}||Q^{(1)}_r-Q^{(1)}_{2r}||_{L^\infty(Q_r\cap \{d>r/2\})}\\
	\leq & Cr^{-1-s}||Q^{(1)}_rd^s-Q^{(1)}_{2r}d^s||_{L^\infty(Q_r)}\\
	\leq & Cr^{-1-s+\alpha}\left[ Q^{(1)}_rd^s-Q^{(1)}_{2r}d^s\right]_{C^\alpha_p(Q_r)}\\
	\leq & Cr^{-1-s+\alpha}\bigg( \left[u_1-q_ru_2- Q^{(1)}_rd^s\right]_{C^\alpha_p(Q_r)} + \left[u_1-q_{2r}u_2 -  Q^{(1)}_{2r}d^s\right]_{C^\alpha_p(Q_r)}\\
	&\quad\quad\quad\quad\quad + \left[ (q_r-q_{2r})u_2\right]_{C^\alpha_p(Q_r)}\bigg)\\
	\leq & CC_0r^{-1-s+\alpha+a} + Cr^{-1}|q_r-q_{2r}|.
	\end{align*}
	Moreover, we have
	\begin{align*}
	||(q_r-q_{2r})u_2 - (Q^{(1)}_r-Q^{(1)}_{2r})d^s||_{L^\infty(Q_r)} 
	\leq & r^\alpha\left[ (q_r-q_{2r})u_2 - (Q^{(1)}_r-Q^{(1)}_{2r})d^s\right]_{C^\alpha_p\infty(Q_r)} \\
	\leq & C_0r^{\alpha+a},
	\end{align*}
	and hence also
	$$||(q_r-q_{2r})u_2 - (Q^{(1)}_r-Q^{(1)}_{2r})d^s||_{L^\infty(Q_r\cap \{d>r/2\})} \leq C_0r^{\alpha+a}.$$
	It follows from \cite[Lemma A.11]{AR20}, that $$|q_r-q_{2r}|\leq CC_0 r^{\alpha+a-s}.$$
	Combining it with the first inequality, we also obtain
	$$|Q^{(1)}_r-Q^{(1)}_{2r}|\leq CC_0 r^{\alpha+a-1-s}.$$
	In the same way as in \cite[Lemma 4.5]{AR20} this implies the existence of the limits $q_0$ and $Q^{(1)}_0$, together with estimates
	$$|Q^{(1)}_0-Q^{(1)}_r|\leq CC_0 r^{\alpha+a-1-s},\quad\quad |q_0-q_r|\leq CC_0 r^{\alpha+a-s},$$
	and so 
	\begin{align*}
	\left[ u_1-q_0u_2-Q^{(1)}_0d^s\right]_{C^\alpha_p(Q_r)}\leq & \left[ u_1-q_ru_2-Q^{(1)}_rd^s\right]_{C^\alpha_p(Q_r)}\\
	&+ \left[( q_0-q_r)u_2\right]_{C^\alpha_p(Q_r)} + \left[ (Q^{(1)}_0-Q^{(1)}_r)d^s\right]_{C^\alpha_p(Q_r)}\\
	\leq & CC_0r^{a}.
	\end{align*}
	Therefore the claim is proved.
\end{proof}

\begin{lemma}\label{lem:A9fake}
	Let $x\in B_{1/2}$ with $x_n>0$, $p\in(0,2s)$ and $q>\max\{1-p,0\}$.
	Then
	$$\int_{B_1}y_{n+}^{p-1}|x-y|^{-n+q}dy\leq C_\varepsilon x_n^{-\varepsilon},$$
	for any $\varepsilon>0$. 
\end{lemma}
\begin{proof}
	Denote $r=\frac{x_n}{2}$ and split the integral into $B_r(x)$ and the complement. In the integral with the pole we estimate $y_{n}^{p-1}$ with $r^{p-1}$ and then integrate the pole to obtain another $Cr^q$. Thanks to the assumptions this is bounded.
	
	The remaining part is handled as follows. Choose $\varepsilon>0$ and estimate
	$$\int_{B_1\backslash B_r(x)}y_{n+}^{p-1}|x-y|^{-n+q}dy\leq
	\int_{B_1\backslash B_r(x)}y_{n+}^{p-1}|x-y|^{-n+1-p-\varepsilon}dy ,$$
	since $|x-y|\leq 1$ and $q> 1-p$. Rescaling the integral on the right-hand side we get
	$$r^{-\varepsilon}\int_{B_{1/r}\backslash B_1(x/r)}y_{n+}^{p-1}\left|\frac{x}{r}-y\right|^{-n+1-p-\varepsilon}dy\leq C_\varepsilon x_n^{-\varepsilon},$$
	as stated, since the $n-$th coordinate of $\frac{x}{r}=2$.
\end{proof}


\begin{thebibliography}{99} 
		
		\bibitem{AR20} N. Abatangelo, X. Ros-Oton, \textit{Obstacle problems for integro-differential operators: Higher regularity of free boundaries}, Adv. Math. \textbf{360} (2020), 106931, 61pp.
		
		\bibitem{BFR18} B. Barrios, A. Figalli, X. Ros-Oton,  \textit{Free boundary regularity in the parabolic fractional obstacle problem,} Comm. Pure Appl. Math. \textbf{71} (2018), 2129--2159. 
		
		\bibitem{BWZ17} U. Biccari,  M. Warma, E. Zuazua, \textit{Local regularity for fractional heat equations}, SEMA SIMAI Springer Series. \textbf{17}  (2017).
		
		\bibitem{CF13}  L. Caffarelli,  A. Figalli,  \textit{Regularity of solutions to the parabolic fractional obstacle problem},
		J. Reine Angew. Math. \textbf{680} (2013), 191--233.
		
		\bibitem{CRS17} L. Caffarelli, X. Ros-Oton, J. Serra, \textit{Obstacle problems for integro-differential operators: regularity of solutions
		and free boundaries}, Invent. Math. \textbf{208} (2017), 1155--1211.
		
		\bibitem{CT04} R. Cont, P. Tankov,  \textit{Financial Modelling with Jump Processes}, Chapman \& Hall/CRC Financial
		Mathematics Series, Chapman \& Hall/CRC, Boca Raton, Fla., (2004).
		
		\bibitem{CSS08} L. Caffarelli, S. Salsa, L. Silvestre, \textit{Regularity estimates for the solution and the free boundary of the obstacle
		problem for the fractional Laplacian}, Invent. Math. \textbf{171} (2008), 425--461.
		
		\bibitem{DS15} D. De Silva, O. Savin, \textit{A note on higher regularity boundary Harnack inequality}, Disc. Cont. Dyn. Syst. \textbf{35}
		(2015), 6155--6163.
		
		\bibitem{DS16} D. De Silva, O. Savin, \textit{Boundary Harnack estimates in slit domains and applications to
			thin free boundary problems}, Rev. Mat. Iberoam. \textbf{32} (2016), 891--912.
		
		\bibitem{FR17} X. Fernandez-Real, X. Ros-Oton, \textit{Regularity theory for general stable operators: parabolic equations,} J. Funct. Anal. \textbf{272} (2017), 4165--4221.
		
		\bibitem{FR20} X. Fern\'andez-Real, X. Ros-Oton, \textit{Regularity Theory for Elliptic PDE}, forthcoming book (2022), available at the webpage of the authors.
		
		\bibitem{FR21} X. Fernandez-Real, X. Ros-Oton, \textit{Free boundary regularity for almost every solution to the Signorini problem}, Arch. Rat. Mech. Anal. \textbf{240} (2021), 419--466.
		
		\bibitem{FRS22} A. Figalli, X. Ros-Oton, J. Serra, \textit{Regularity for parabolic nonlocal obstacle problems: a new and unified approach}, in preparation, (2022).
		
		\bibitem{F20} M. Focardi, E. Spadaro, \textit{The local structure of the free boundary in the fractional obstacle problem}, Adv. Calc. Var. \textbf{15} (2022), 323–-349.
		
		\bibitem{GR19} N. Garofalo, X. Ros-Oton, \textit{Structure and regularity of the singular set in the obstacle problem for the fractional
		Laplacian}, Rev. Mat. Iberoam. \textbf{35} (2019), 1309--1365.
		
		\bibitem{G19} G. Grubb, \textit{Limited regularity of solutions to fractional heat and Schr\"odinger equations},  Disc. Cont. Dyn. Syst. \textbf{39} (2019), 3609–-3634.
		
		\bibitem{JN17} Y. Jhaveri, R. Neumayer, \textit{Higher regularity of the free boundary in the obstacle problem for the fractional
		Laplacian}, Adv. Math. \textbf{311} (2017), 748-795.
		
		\bibitem{KRS19} H. Koch, A. R\"uland, W. Shi, \textit{Higher regularity for the fractional thin obstacle problem}, New York J. Math.  \textbf{25} (2019) 745–838.
		
		\bibitem{K21b} T. Kukuljan, \textit{Higher order parabolic boundary Harnack inequality}, Disc. Cont. Dyn. Syst. \textbf{42} (2022).
		
		\bibitem{K21a} T. Kukuljan, \textit{The fractional obstacle problem with drift: higher regularity of free boundaries}, J. Funct. Anal. \textbf{281} (2021), 109114, 60pp.
		
		\bibitem{RT21} X. Ros-Oton, D. Torres-Latorre, \textit{Optimal regularity for supercritical parabolic obstacle problems}, 
		preprint arXiv (2021).
		
		\bibitem{RV18} X. Ros-Oton, H. Vivas, \textit{Higher-order boundary regularity estimates for nonlocal parabolic equations}, Calc. Var. Partial Differential Equations \textbf{57} (2018), 111, 20 pp.
		
		\bibitem{S07} L. Silvestre, \textit{Regularity of the obstacle problem for a fractional power of the Laplace operator}, Comm. Pure Appl. Math. \textbf{60} (2007), 67--112.
	\end{thebibliography}
\end{document}